\documentclass[11pt]{amsart}

\usepackage[T1]{fontenc}
\usepackage[utf8]{inputenc}

\usepackage{amsmath,amssymb,amsthm,mathtools}
\usepackage{mathrsfs}
\usepackage{tikz-cd}
\usepackage{enumitem}

\usepackage[colorlinks=true,
  linkcolor=blue,
  citecolor=blue,
  urlcolor=blue]{hyperref}

\newtheorem{theorem}{Theorem}[section]
\newtheorem{proposition}[theorem]{Proposition}
\newtheorem{lemma}[theorem]{Lemma}
\newtheorem{corollary}[theorem]{Corollary}
\newtheorem{fact}[theorem]{Fact}

\theoremstyle{definition}
\newtheorem{definition}[theorem]{Definition}
\newtheorem{example}[theorem]{Example}

\theoremstyle{remark}
\newtheorem{remark}[theorem]{Remark}

\newcommand{\cl}{\operatorname{cl}}

\newcommand{\st}{\operatorname{st}}
\newcommand{\ord}{\operatorname{ord}}

\newcommand{\RCF}{\operatorname{RCF}}
\newcommand{\CODF}{\operatorname{CODF}}
\newcommand{\Jet}{\operatorname{Jet}}

\makeatletter
\def\subsection{\@startsection{subsection}{2}%
  \z@{.7\linespacing\@plus.3\linespacing}{.4\linespacing}%
  {\normalfont\bfseries}}
\makeatother

\title{\texorpdfstring{$\delta$}{Delta}-Cell Decomposition and Curve Selection}

\author{Xiaoduo Wang}

\address{Department of Mathematics, University of Manchester}
\email{xiaoduo.wang@manchester.ac.uk}

\date{\today}

\begin{document}

\begin{abstract}
We develop a cell decomposition framework for o-minimal structures equipped with a generic derivation. To a $\delta$-cell we associate source cells and finite configurations in ordinary o-minimal sorts, allowing differential-topological questions to be studied through finite jet spaces. We then introduce a metric space of definable curve germs and identify its local half-space pieces with Cartesian powers of the maximal ideal of the Hardy field of definable germs. Using this germ-space description, we prove an abstract curve selection theorem for the $\delta$-topology. In the case of closed ordered differential fields, we further describe concrete asymptotic representatives for the abstract curve germs.
\end{abstract}

\maketitle

\section{Introduction}

O-minimal structures provide a setting in which definable sets admit a geometric theory close to semialgebraic and subanalytic geometry. Basic tameness properties such as cell decomposition, dimension theory, monotonicity, and curve selection make it possible to study definable sets by geometric methods; see, for instance, \cite{PillaySteinhorn1986,vandenDries2003}. A natural question is what remains of this geometry after adding a derivation.

The first and most classical model-theoretic example is the theory of closed ordered differential fields, introduced by Singer in \cite{Singer1978}. This theory may be viewed as a differential analogue of real closed fields. Later work studied its definable sets and developed versions of cell decomposition and dimension theory; see, for example, \cite{BrihayeMichauxRivière2009,Point2011}. Related developments include the treatment of several commuting derivations in \cite{Rivière2006}, and the study of tame pairs of closed ordered differential fields in \cite{Borrata2021}.

A more general framework was introduced by Fornasiero and Kaplan in \cite{FornasieroKaplan2020}. Let \(T\) be an o-minimal theory expanding the theory of ordered fields. A \(T\)-derivation is a derivation which satisfies the chain rule with respect to all \(\mathcal L(\varnothing)\)-definable \(\mathcal C^1\)-functions. Under suitable hypotheses on \(T\), the theory of models of \(T\) equipped with a \(T\)-derivation admits a model companion, denoted \(T_g^\delta\). In models of \(T_g^\delta\), the derivation behaves generically. This construction recovers \(\CODF\) when \(T\) is the theory of real closed fields, and extends the differential setting to arbitrary o-minimal theories. Further work in this direction includes \cite{FornasieroTerzo2024,CubidesPoint2023}.

The presence of a derivation changes the topology one should consider. In ordinary o-minimal geometry, definable sets are studied with respect to the Euclidean topology. In the differential setting, one also has the \(\delta\)-topology. In the case of closed ordered differential fields, this topology was introduced by Brihaye, Michaux and Rivière in \cite{BrihayeMichauxRivière2009}. It can be described as the topology generated by the finite jet maps, or equivalently as the coarsest topology refining the original topology for which the derivation is continuous. The same construction applies in models of \(T_g^\delta\). Although this topology is natural, it is not itself definable in the differential language. Consequently, even basic questions about closure and curve selection become more subtle than in the o-minimal case.

The difficulty is that a definable condition of finite differential order gives rise to infinitely many finite jet conditions. At each finite level one obtains an ordinary \(\mathcal L\)-definable set in an o-minimal structure, but closure in the \(\delta\)-topology depends on all finite levels simultaneously. Thus, instead of working only with the original differential set, one is led to study the sequence of its finite configurations in jet spaces.

The first purpose of this paper is to develop a systematic cell-theoretic language for this situation. Starting from the \(\delta\)-cell decomposition of \cite{BrihayeMichauxRivière2009} and its generalizations in \cite{FornasieroKaplan2020,CubidesPoint2023}, we introduce source cells and configurations associated to a \(\delta\)-cell. A source cell is an ordinary o-minimal cell in a finite jet space which records the finite differential data defining the \(\delta\)-cell. Higher configurations are obtained by repeatedly applying the chain rule. In this way, a differential cell is replaced by a compatible system of ordinary o-minimal cells.

This construction is useful because it separates two kinds of information. The differential structure is encoded by the way configurations are lifted from one finite level to the next, while the geometry at each finite level is ordinary o-minimal geometry. We prove a refined \(\delta\)-decomposition theorem, Theorem~\ref{Refined delta decomposition theorem}, in which source cells may be chosen smooth and with dense realized jet space. We then construct canonical source cells by passing to minimal prepared source layouts. The main point of this construction is that configurations become independent of the choice of defining formula; see Theorem~\ref{Uniqueness of minimal source configurations}. We also prove that source configurations behave well under projections and fibers, Theorem~\ref{Projection and fibers of a delta cell is a delta cell}.

The second purpose of the paper is to formulate curve selection at the correct level. In the o-minimal setting, curve selection says that a point in the closure of a definable set can be approached by a definable curve contained in the set. In the \(\delta\)-topology, a naive version of this statement is not the right object: a curve should not only approach the point at the original level, but should also be compatible with all higher finite configurations. We therefore define \(\delta\)-curves as compatible sequences of ordinary definable curves, one in each finite configuration. The finite-configuration characterization of \(\delta\)-closure, Lemma~\ref{Finite configurations characterize delta closure}, is the bridge between the topology and these sequences.

The main new ingredient is a metric structure on the space of ordinary curve germs. Fix a point \(a\). We consider definable curves converging to \(a\), and identify two such curves if their images agree in a sufficiently small neighbourhood of \(a\). The distance between two curve germs is measured by the rate at which their images approach one another in the Hausdorff metric. This rate is naturally an element of the Hardy field of definable germs. Hardy fields and their connection with o-minimality are discussed, for example, in \cite{MillerHardy,ADH2019}. After restricting to a half-space piece, the space of curve germs is identified with a Cartesian power of the maximal ideal of this Hardy field; see Proposition~\ref{Charts are Euclidean spaces over the Hardy field}. Thus the local topology of the germ space becomes an ordinary product topology over a Hardy field.

This germ-space construction is the point at which the curve selection problem becomes manageable. Instead of attempting to choose one concrete definable curve satisfying all infinitely many finite-level conditions at once, we work with an abstract germ. The abstract germ records a limiting direction in the Hardy-field coordinate space. Near this abstract germ, one can find ordinary definable curve germs whose induced canonical \(\delta\)-curves satisfy any prescribed finite number of convergence conditions. This gives the abstract curve selection theorem for the \(\delta\)-topology, Theorem~\ref{Abstract curve selection lemma}.

In the final section we specialize to closed ordered differential fields over the real field. In this case the abstract germs obtained above can be compared with explicit asymptotic representatives. We use Puiseux and Hahn-field methods, in the spirit of \cite{MarkerMessmerPillay1996,Kaplansky1942,Kaplansky1944}, to represent the relevant cuts by concrete real germs. The resulting statement, Theorem~\ref{Concrete realization of abstract CODF curve}, gives concrete representatives for the abstract curve germs under the stated real-germ interpretation hypothesis, and separates the algebraic Puiseux-type cases from the generalized power series case.

We now describe the organization of the paper. Section~2 recalls the necessary background and fixes notation. Section~3 introduces finite jet spaces and the \(\delta\)-topology. Section~4 defines \(\delta\)-cells and proves the refined \(\delta\)-decomposition theorem. Section~5 develops configurations and canonical source cells. Section~6 defines \(\delta\)-curves and canonical \(\delta\)-curves. Section~7 constructs the metric space of curve germs and identifies its local pieces with Hardy-field coordinate spaces. Section~8 proves the abstract curve selection theorem. Section~9 treats concrete representatives in \(\CODF\).

\section{Preliminaries}

We recall the background and notation used throughout the paper. Most of the material in this section is standard; it is included to fix conventions for o-minimality, generic derivations, Hardy fields, and convex pairs.

\subsection{O-minimality and weak o-minimality}

Let \(\mathcal M=(M,<,\ldots)\) be an expansion of a dense linear order without endpoints. We say that \(\mathcal M\) is \emph{o-minimal} if every definable subset of \(M\) is a finite union of points and open intervals. A complete theory extending the theory of ordered fields is called o-minimal if all of its models are o-minimal. Equivalently, by the o-minimality theorem, such a theory is o-minimal if it has an o-minimal model.

Throughout the paper, \(T\) is a complete o-minimal theory expanding the theory of real closed ordered fields in a language \(\mathcal L\). We assume that \(T\) has quantifier elimination and is universally axiomatizable. This is the usual convention in the study of \(T\)-convexity and \(T\)-derivations. Indeed, o-minimal expansions of real closed fields have definable Skolem functions, and one may replace \(\mathcal L\) by the language \(\mathcal L^{\mathrm{df}}\) obtained by adding a function symbol for every \(\mathcal L(\varnothing)\)-definable function. The corresponding definitional expansion \(T^{\mathrm{df}}\) has quantifier elimination and is universally axiomatizable. Passing to this definitional expansion does not change the definable sets. Thus, as in \cite{vandendriesLewenberg1995,vandenDries1997,FornasieroKaplan2020}, we work under these hypotheses from the beginning. Unless stated otherwise, definable means definable with parameters.

We use standard facts from o-minimality without further comment, including cell decomposition, monotonicity, definable choice, dimension theory, and the o-minimal curve selection lemma. These facts can be found in \cite{vandenDries2003}. When smoothness is required, we use \(\mathcal C^p\)-cell decomposition. In the sections on curve selection, we assume that \(T\) admits \(\mathcal C^\infty\)-cell decomposition.

We shall also use the corresponding Nash form of cell decomposition in the semialgebraic case. By a \emph{Nash function} we mean a semialgebraic \(\mathcal C^\infty\)-function, and a \emph{Nash cell} is a semialgebraic cell whose defining functions are Nash. Every finite family of semialgebraic sets admits a finite Nash cell decomposition compatible with the family; see \cite{BochnakCosteRoy1998}.

We next recall the small amount of weak o-minimality needed later. Let \(\mathcal N=(N,<,\ldots)\) be a linearly ordered structure. We say that \(\mathcal N\) is \emph{weakly o-minimal} if every definable subset of \(N\) is a finite union of convex subsets of \(N\). A theory is weakly o-minimal if all of its models are weakly o-minimal.

The weakly o-minimal structures used below arise from \(T\)-convex pairs. We recall the particular form needed for Hardy fields in Section~\ref{Subsection Hardy fields and convex pairs}.

\subsection{Generic derivations}

Let \(\mathcal L^\delta:=\mathcal L\cup\{\delta\}\), where \(\delta\) is a unary function symbol. We recall the basic facts on generic derivations from \cite{FornasieroKaplan2020}.

Let \((\mathcal M,\delta)\) be an \(\mathcal L^\delta\)-structure with \(\mathcal M\models T\). We say that \(\delta\) is a \emph{\(T\)-derivation} on \(\mathcal M\) if, for every \(\mathcal L(\varnothing)\)-definable \(\mathcal C^1\)-function \(f:U\to M\), with \(U\subseteq M^n\) open, one has \(\delta f(u)=\mathbf J_f(u)\delta u\) for all \(u\in U\). Here \(\mathbf J_f\) denotes the Jacobian matrix of \(f\), computed in the usual o-minimal differential calculus over the underlying real closed field. We denote by \(T^\delta\) the \(\mathcal L^\delta\)-theory extending \(T\) by the axioms asserting that \(\delta\) is a \(T\)-derivation.

For definable functions with parameters, the chain rule has the following form.

\begin{fact}\cite[Lemma~2.12]{FornasieroKaplan2020}\label{Chain rule for definable functions}
Suppose that \((\mathcal M,\delta)\models T^\delta\). Let \(k>0\), and let \(f\) be an \(\mathcal L(M)\)-definable \(\mathcal C^k\)-function on an open set \(U\subseteq M^n\). Then there exists a unique \(\mathcal L(M)\)-definable \(\mathcal C^{k-1}\)-function \(f^{[\delta]}:U\to M\) such that
\[
\delta f(u)=f^{[\delta]}(u)+\mathbf J_f(u)\delta u
\]
for all \(u\in U\). Moreover, if \(f\) is \(\mathcal L(A)\)-definable for some \(A\subseteq\ker(\delta)\), then \(f^{[\delta]}=0\).
\end{fact}

We now recall the genericity condition. For \(a\in M\) and \(n\in\mathbb N\), write
\[
\Jet_n(a):=(a,\delta a,\ldots,\delta^n a).
\]
If \(a=(a_1,\ldots,a_m)\in M^m\) and \(n\in\mathbb N\), write
\[
\Jet_{(n,\ldots,n)}(a):=(\Jet_n(a_1),\ldots,\Jet_n(a_m)).
\]

A \(T\)-derivation \(\delta\) is called \emph{generic} if, for every \(n\in\mathbb N\) and every \(\mathcal L(M)\)-definable set \(A\subseteq M^{n+1}\), if \(\dim(\Pi_n(A))=n\), then there is \(a\in M\) such that \(\Jet_n(a)\in A\). We denote by \(T_g^\delta\) the \(\mathcal L^\delta\)-theory extending \(T^\delta\) by this genericity scheme.

\begin{fact}\cite[Theorem~4.8]{FornasieroKaplan2020}\label{Model companion of T derivations}
The theory \(T_g^\delta\) is the model completion of \(T^\delta\). Since \(T\) has quantifier elimination and is universally axiomatizable, \(T_g^\delta\) has quantifier elimination.
\end{fact}

When \(T=\RCF\), the theory \(T_g^\delta\) is \(\CODF\).

The following two facts are used repeatedly to pass between differential definability and ordinary definability in finite jet spaces.

\begin{fact}\cite[Lemma~4.11]{FornasieroKaplan2020}\label{Forget the derivation}
For every \(\mathcal L^\delta\)-formula \(\varphi(x)\), possibly with parameters, there exist \(m\in\mathbb N\) and an \(\mathcal L\)-formula \(\widetilde\varphi\) such that
\[
T_g^\delta\vdash \forall x\bigl(\varphi(x)\leftrightarrow \widetilde\varphi(\Jet_m(x))\bigr).
\]
\end{fact}

\begin{fact}\cite[Lemma~5.5]{FornasieroKaplan2020}\label{Jet space is dense}
For every \((\mathcal M,\delta)\models T_g^\delta\), and all \(m,n\in\mathbb N\), the set \(\Jet_{(m,\ldots,m)}(M^n)\) is dense in \(M^{n(m+1)}\).
\end{fact}

We shall also use that \(T\) is the open core of \(T_g^\delta\).

\begin{fact}\cite[Proposition~5.12]{FornasieroKaplan2020}\label{Open core}
Every open \(\mathcal L^\delta(A)\)-definable subset of \(M^n\) is \(\mathcal L(A)\)-definable.
\end{fact}

Consequently, if \(X\subseteq M^n\) is \(\mathcal L^\delta(A)\)-definable, then its Euclidean closure and Euclidean interior are \(\mathcal L(A)\)-definable.

\subsection{Hardy fields and convex pairs}\label{Subsection Hardy fields and convex pairs}

We recall the Hardy-field notation used in the construction of the germ space. Let \(\mathcal M\) be an o-minimal expansion of a real closed field, and let \(a\in M\). A germ at \(a^+\) is an equivalence class of definable functions \(f:(a,b)\to M\), where two functions are equivalent if they agree on some interval \((a,c)\). We denote the germ of \(f\) by \([f]\).

Let \(\mathcal H_{a^+}\) be the set of germs at \(a^+\) of unary \(\mathcal M\)-definable functions. By o-minimality, after restricting the domain if necessary, definable functions are continuous and the usual field operations on representatives are well-defined on germs. Thus \(\mathcal H_{a^+}\) is an ordered field, called the Hardy field of definable germs at \(a^+\). If \(a\) is fixed, we write \(\mathcal H\) instead of \(\mathcal H_{a^+}\).

The constant germs identify \(M\) with an elementary substructure of \(\mathcal H\). More precisely, the map \(c\mapsto[c]\), where \([c]\) is the germ of the constant function with value \(c\), is an elementary embedding \(\mathcal M\preccurlyeq\mathcal H\). We shall usually identify \(M\) with its image in \(\mathcal H\). See \cite{MillerHardy}.

Let \(\mathcal K\) be a Hardy field extending \(M\). We write \(\mathcal H(M)\) for the convex hull of \(M\) in \(K\), namely
\[
\mathcal H(M):=\{x\in K: |x|\leq c \text{ for some } c\in M^{>0}\}.
\]
Its maximal ideal is
\[
\mathfrak m_{\mathcal K}:=\{x\in K: |x|<c \text{ for every } c\in M^{>0}\}.
\]
Equivalently, \(\mathfrak m_{\mathcal K}\) is the set of infinitesimals of \(K\) over \(M\).

We shall use the pair \((K,\mathcal H(M))\), where \(\mathcal H(M)\) is named by a unary predicate. This pair is an example of a \(T\)-convex structure: the predicate names a convex subring which is closed under all continuous \(\mathcal L(\varnothing)\)-definable functions. The theory of such pairs is denoted by \(T_{\mathrm{convex}}\). We use the standard facts that \(T_{\mathrm{convex}}\) is complete, admits quantifier elimination in the language \(\mathcal L_{\mathrm{convex}}\), and is weakly o-minimal; see \cite{vandendriesLewenberg1995,vandenDries1997}. This is the structure in which the definable neighbourhoods of abstract curve germs will be taken.

\section{Jets and \texorpdfstring{$\delta$}{Delta}-Cells}

Throughout this section, fix \((\mathcal M,\delta)\models T_g^\delta\). We begin by recalling the finite jet spaces and the topology induced by them. These notions provide the ambient o-minimal spaces in which source cells and configurations will live.

\subsection{Finite jets and the \texorpdfstring{$\delta$}{delta}-topology}

Let \(k>0\) and let \(n=(n_1,\ldots,n_k)\in\mathbb N^k\). Set
\[
E_n:=\prod_{i=1}^k M^{n_i+1}.
\]
For \(a=(a_1,\ldots,a_k)\in M^k\), define
\[
\Jet_n(a):=(a_1,\delta a_1,\ldots,\delta^{n_1}a_1,\ldots,a_k,\delta a_k,\ldots,\delta^{n_k}a_k)\in E_n.
\]
When \(k=1\), we write \(\Jet_m(a)\) instead of \(\Jet_{(m)}(a)\).

We order \(\mathbb N^k\) coordinatewise. Thus \(m\leq n\) means that \(m_i\leq n_i\) for all \(i=1,\ldots,k\). If \(m\leq n\), let \(\Pi_{m,n}:E_n\to E_m\) be the coordinate projection forgetting the derivatives of order \(>m_i\) in the \(i^{\mathrm{th}}\) block. Then \(\Pi_{m,n}\circ\Jet_n=\Jet_m\). When the ambient order \(n\) is clear, we write \(\Pi_m\) instead of \(\Pi_{m,n}\).

The finite jet space \(E_n\) is equipped with the sup metric
\[
d_{(0,\ldots,0)}(x,y):=\max |x_\ell-y_\ell|,
\]
where the maximum is taken over all coordinates of \(E_n\). Pulling this metric back along \(\Jet_n\), we obtain a metric \(d_n\) on \(M^k\), namely
\[
d_n(a,b):=\max_{1\leq i\leq k}\max_{0\leq j\leq n_i}|\delta^j a_i-\delta^j b_i|.
\]
Let \(\mathcal T_n\) be the topology induced by \(d_n\). Equivalently, \(\mathcal T_n\) is the pullback along \(\Jet_n\) of the Euclidean topology on \(E_n\).

If \(m\leq n\), then \(\mathcal T_m\subseteq\mathcal T_n\). Since \(\mathbb N^k\) is directed, the union
\[
\mathcal T_\delta:=\bigcup_{n\in\mathbb N^k}\mathcal T_n
\]
is a topology on \(M^k\). We call it the \(\delta\)-topology. Equivalently, \(\mathcal T_\delta\) is the topology generated by the finite jet maps. In particular, it is the coarsest topology refining the original o-minimal topology for which the derivation \(\delta\) is continuous.

The following finite-jet description of definable \(\delta\)-open sets is the analogue, in the present setting, of Proposition~3.3 of \cite{BrihayeMichauxRivière2009}. The proof is the same, using the density of finite jet spaces from Fact~\ref{Jet space is dense}.

\begin{fact}\label{Finite jet characterization of delta topology}
Let \(X\subseteq M^k\) be \(\mathcal L^\delta(M)\)-definable. Then \(X\) is open in the \(\delta\)-topology if and only if there are \(n\in\mathbb N^k\) and an \(\mathcal L(M)\)-definable Euclidean open set \(U\subseteq E_n\) such that \(X=\Jet_n^{-1}(U)\).
\end{fact}

For \(X\subseteq M^k\), let \(\cl_n(X)\) denote the closure of \(X\) with respect to \(\mathcal T_n\), and let \(\cl_\delta(X)\) denote the closure with respect to \(\mathcal T_\delta\). Then
\[
a\in\cl_n(X)
\quad\text{if and only if}\quad
\Jet_n(a)\in\cl(\Jet_n(X)),
\]
where the closure on the right is Euclidean closure in \(E_n\). Since \(\mathcal T_m\subseteq\mathcal T_n\) whenever \(m\leq n\), the corresponding closures satisfy \(\cl_n(X)\subseteq\cl_m(X)\).

\subsection{\texorpdfstring{$\delta$}{Delta}-cells}

Let \(X\subseteq M^k\) be \(\mathcal L^\delta(B)\)-definable. By Fact~\ref{Forget the derivation}, after increasing the order if necessary, there are \(n\in\mathbb N^k\) and an \(\mathcal L(B)\)-definable set \(X^{\mathcal L}\subseteq E_n\) such that \(\Jet_n(X)=X^{\mathcal L}\cap\Jet_n(M^k)\). Thus an \(\mathcal L^\delta\)-definable set can be represented by an ordinary \(\mathcal L\)-definable set in a finite jet space.

We now recall the cell-theoretic version of this representation. The terminology in the literature is slightly weaker than the one used in this paper. In \cite{BrihayeMichauxRivière2009}, a \(\delta\)-cell is defined by requiring that, for some finite jet level, the corresponding source set is an ordinary cell in the finite jet space. For our purposes, we strengthen this definition by requiring that the realized jet space be dense in the chosen source cell. This density condition ensures that the source cell records only the part of the finite jet space which is approached by actual jets.

Let \(\varphi(x)\) be a quantifier-free \(\mathcal L^\delta(B)\)-formula, where \(x=(x_1,\ldots,x_k)\). For each \(i=1,\ldots,k\), let \(n_i\) be the largest integer such that \(\delta^{n_i}x_i\) occurs in \(\varphi\), and set \(n_i=0\) if no derivative of \(x_i\) occurs. We write \(\ord(\varphi):=(n_1,\ldots,n_k)\).

Let \(n=\ord(\varphi)\). By replacing each occurrence of \(\delta^j x_i\) in \(\varphi\) by the corresponding coordinate \(x_i^{(j)}\) of \(E_n\), we obtain an \(\mathcal L(B)\)-formula in the jet variables. If \(C\subseteq M^k\) is defined by \(\varphi\), we denote the subset of \(E_n\) defined by this \(\mathcal L(B)\)-formula by \(C_\varphi^{\mathcal L}\).

Let \(C\subseteq M^k\) be \(\mathcal L^\delta(B)\)-definable. A formula \(\varphi(x)\) defining \(C\) is called a \emph{source formula} for \(C\) if, writing \(n=\ord(\varphi)\), the set \(C_\varphi^{\mathcal L}\subseteq E_n\) is an o-minimal cell and \(\Jet_n(C)\) is Euclidean dense in \(C_\varphi^{\mathcal L}\). In this case \(C_\varphi^{\mathcal L}\) is called the \emph{source cell of \(C\) with respect to \(\varphi\)}. A set \(C\subseteq M^k\) is called a \emph{\(\delta\)-cell} if it admits a source formula.

A finite collection \(\mathcal D\) of \(\delta\)-cells in \(M^k\) is called a \emph{\(\delta\)-decomposition} of \(M^k\) if it is a partition of \(M^k\) and is compatible with coordinate projections in the following inductive sense. If \(k=1\), this means simply that \(\mathcal D\) is a finite partition of \(M\) into \(\delta\)-cells. If \(k>1\), then
\[
\Pi_{k-1}(\mathcal D):=\{\Pi_{k-1}(D):D\in\mathcal D\}
\]
is, after removing repetitions, a \(\delta\)-decomposition of \(M^{k-1}\), and for each \(D\in\mathcal D\), the projection \(\Pi_{k-1}(D)\) belongs to \(\Pi_{k-1}(\mathcal D)\).

If \(A_1,\ldots,A_l\subseteq M^k\), we say that \(\mathcal D\) is \emph{compatible} with \(A_1,\ldots,A_l\) if, for every \(D\in\mathcal D\) and every \(i=1,\ldots,l\), either \(D\subseteq A_i\) or \(D\cap A_i=\varnothing\).

The following theorem is a refinement, in the o-minimal setting, of the \(\delta\)-cell decomposition results mentioned above. Cubides Kovacsics and Point prove a \(\delta\)-cell decomposition theorem for topological fields with a generic derivation, and in \cite[Remark~3.2.4]{CubidesPoint2023} they observe that, in the \(\CODF\) case, their argument may be carried out with o-minimal cells. The refinement below keeps track of the parameter set, compatibility with a prescribed finite family, and the smoothness of the source cells. In particular, when \(T\) admits \(\mathcal C^p\)-cell decomposition, the source cells may be chosen of class \(\mathcal C^p\).

\begin{theorem}[Refined \(\delta\)-decomposition theorem]\label{Refined delta decomposition theorem}
Let \((\mathcal M,\delta)\models T_g^\delta\), let \(B\subseteq M\), and fix \(p\in\mathbb N\). For every finite collection \(\mathcal A=\{A_1,\ldots,A_l\}\) of \(\mathcal L^\delta(B)\)-definable subsets of \(M^k\), there is a finite \(\delta\)-decomposition \(\mathcal D\) of \(M^k\), definable over \(B\), compatible with \(\mathcal A\), such that every \(D\in\mathcal D\) admits a source formula \(\varphi\) whose source cell \(D_\varphi^{\mathcal L}\) is of class \(\mathcal C^p\).
\end{theorem}

\begin{proof}
Choose \(\mathcal L^\delta(B)\)-formulas \(\varphi_1,\ldots,\varphi_l\) defining \(A_1,\ldots,A_l\), respectively. After increasing orders if necessary, we may assume that there is \(n\in\mathbb N\) such that \(\ord(\varphi_i)\leq(n,\ldots,n)\) for every \(i=1,\ldots,l\). Set \(N=k(n+1)\), and let
\[
\mathcal A^{\mathcal L}:=\{(A_1)_{\varphi_1}^{\mathcal L},\ldots,(A_l)_{\varphi_l}^{\mathcal L}\}.
\]
By the o-minimal \(\mathcal C^p\)-cell decomposition theorem, there is a \(\mathcal C^p\)-cell decomposition \(\mathcal D^{\mathcal L}\) of \(M^N\), definable over \(B\), compatible with \(\mathcal A^{\mathcal L}\).

We refine \(\mathcal D^{\mathcal L}\) so that the realized jet space is dense in every cell which meets it. Let \(D\in\mathcal D^{\mathcal L}\). If \(D\cap\Jet_{(n,\ldots,n)}(M^k)\) is dense in \(D\), there is nothing to do. Otherwise set \(C_D:=\cl(D\cap\Jet_{(n,\ldots,n)}(M^k))\cap D\), where the closure is Euclidean closure in \(M^N\). By Fact~\ref{Open core}, \(C_D\) is \(\mathcal L(B)\)-definable. Replace \(D\) by a \(\mathcal C^p\)-cell decomposition of \(C_D\) and of \(D\setminus C_D\). On cells where density still fails, the o-minimal dimension strictly decreases. Repeating this process, we obtain after finitely many steps a \(\mathcal C^p\)-cell refinement, still denoted \(\mathcal D^{\mathcal L}\), such that for every \(D\in\mathcal D^{\mathcal L}\), if \(D\cap\Jet_{(n,\ldots,n)}(M^k)\neq\varnothing\), then this intersection is dense in \(D\).

Define
\[
\mathcal D:=\{\Pi_{(0,\ldots,0)}(D\cap\Jet_{(n,\ldots,n)}(M^k)):D\in\mathcal D^{\mathcal L},\ D\cap\Jet_{(n,\ldots,n)}(M^k)\neq\varnothing\}.
\]
Since \(\mathcal D^{\mathcal L}\) is a cell decomposition in the inductive sense, its projections to initial coordinate blocks again form cell decompositions. Hence \(\mathcal D\) is a finite \(\delta\)-decomposition of \(M^k\). Compatibility with \(\mathcal A\) follows from compatibility of \(\mathcal D^{\mathcal L}\) with \(\mathcal A^{\mathcal L}\).

Finally, each member of \(\mathcal D\) is defined by pulling back a corresponding cell \(D\in\mathcal D^{\mathcal L}\) along \(\Jet_{(n,\ldots,n)}\). The corresponding source cell at order \((n,\ldots,n)\) is \(D\), which is of class \(\mathcal C^p\), and by construction \(D\cap\Jet_{(n,\ldots,n)}(M^k)\) is dense in \(D\). Hence each member of \(\mathcal D\) is a \(\delta\)-cell in the sense above, with a \(\mathcal C^p\) source cell.
\end{proof}

In the case \(T=\RCF\), the source cells may be chosen Nash by Nash cell decomposition.

\begin{corollary}\label{Refined delta decomposition for CODF}
Let \((\mathcal M,\delta)\models\CODF\), and let \(B\subseteq M\). For every finite collection \(\mathcal A=\{A_1,\ldots,A_l\}\) of \(\mathcal L^\delta(B)\)-definable subsets of \(M^k\), there is a finite \(\delta\)-decomposition \(\mathcal D\) of \(M^k\), definable over \(B\), compatible with \(\mathcal A\), such that every \(D\in\mathcal D\) admits a source formula \(\varphi\) whose source cell \(D_\varphi^{\mathcal L}\) is a Nash cell.
\end{corollary}

\subsection{Minimal source cells}

The source cell of a \(\delta\)-cell depends on the chosen source formula. Source formulas may contain redundant derivatives, and different choices of source formula may give source cells in different finite jet spaces. In order to construct configurations later, we first isolate source formulas of minimal order.

We order \(\mathbb N^k\) by reverse lexicographic order, denoted \(\prec_{\mathrm{rlex}}\). Thus, for \(m,n\in\mathbb N^k\), we have \(m\prec_{\mathrm{rlex}} n\) if, at the largest index \(i\) such that \(m_i\neq n_i\), one has \(m_i<n_i\).

Let \(C\subseteq M^k\) be a \(\delta\)-cell. A source formula \(\varphi\) for \(C\) is called \emph{minimal} if \(\ord(\varphi)\) is minimal, among all source formulas for \(C\), with respect to \(\prec_{\mathrm{rlex}}\). A source cell arising from a minimal source formula is called a \emph{minimal source cell}.

In one variable, minimal source cells are canonical. Indeed, the following lemma says that once the order is fixed, the source cell is forced.

\begin{lemma}\label{Source cell is unique in one variable}
Let \(C\subseteq M\) be a \(\delta\)-cell, and let \(\varphi,\psi\) be source formulas for \(C\) of order \(n\). Then \(C_\varphi^{\mathcal L}=C_\psi^{\mathcal L}\).
\end{lemma}

\begin{proof}
Since \(\varphi\) and \(\psi\) define the same set \(C\), we have
\[
C_\varphi^{\mathcal L}\cap\Jet_n(M)=C_\psi^{\mathcal L}\cap\Jet_n(M)=\Jet_n(C).
\]
By assumption, \(\Jet_n(C)\) is dense in both \(C_\varphi^{\mathcal L}\) and \(C_\psi^{\mathcal L}\).

For each \(j=0,\ldots,n\), set \(A_j:=\Pi_j(C_\varphi^{\mathcal L})\) and \(B_j:=\Pi_j(C_\psi^{\mathcal L})\). Then \(A_j\) and \(B_j\) are cells in \(M^{j+1}\). Since projections of cells are open maps onto their images, the density of \(\Jet_n(C)\) in the source cells implies that \(\Jet_j(C)\) is dense in both \(A_j\) and \(B_j\). Hence \(\cl(A_j)=\cl(B_j)\) for each \(j=0,\ldots,n\).

We prove by induction on \(j\) that \(A_j=B_j\). For \(j=0\), the sets \(A_0\) and \(B_0\) are points or open intervals in \(M\). Since they have the same closure, they are equal. Suppose \(j>0\) and \(A_{j-1}=B_{j-1}\). The cells \(A_j\) and \(B_j\) are built over the same base. If one were a graph and the other a band, their closures would have different fiber dimensions over points of the base, contradicting \(\cl(A_j)=\cl(B_j)\). Hence they are either both graphs or both bands. In the graph case, equality of closures forces equality of the defining functions. In the band case, equality of closures forces equality of the two boundary functions. Thus \(A_j=B_j\). Taking \(j=n\), we get \(C_\varphi^{\mathcal L}=C_\psi^{\mathcal L}\).
\end{proof}

Thus, if \(C\subseteq M\) is a one-variable \(\delta\)-cell and \(\varphi\) is a minimal source formula for \(C\), then the source cell \(C_\varphi^{\mathcal L}\) is independent of the chosen minimal source formula. We denote it by \(C^{\mathcal L}\), and call it the \emph{minimal source cell} of \(C\).

\section{Configurations}

In this section we study the finite configurations associated to \(\delta\)-cells. The strengthened definition of \(\delta\)-cell gives control over source cells because the realized jet space is dense in the chosen source cell. Configurations record how these source cells project to lower jet levels and how they extend to higher jet levels under the chain rule.

Unless stated otherwise, all topological notions in this section refer to the Euclidean topology. Thus \(\cl(X)\) denotes Euclidean closure, and open means Euclidean open.

\subsection{The chain-rule construction}

Let \(l>0\), let \(k>0\), and let \(n=(n_1,\ldots,n_k)\in\mathbb N^k\). Let \(U\subseteq E_n\) be an \(\mathcal L(M)\)-definable open set, and let \(f:U\to M\) be an \(\mathcal L(M)\)-definable function of class \(\mathcal C^l\). Let \(\mathcal S(U)\subseteq E_{n+1}\) be the inverse image of \(U\) under the natural projection \(E_{n+1}\to E_n\).

We define an \(\mathcal L(M)\)-definable function
\[
\mathcal S(f):\mathcal S(U)\to M
\]
as follows. For \(z\in\mathcal S(U)\), set
\[
\mathcal S(f)(z)
=
f^{[\delta]}(\Pi_n(z))
+
\mathbf J_f(\Pi_n(z))\cdot v(z),
\]
where
\[
v(z):=(x_1^{(1)},\ldots,x_1^{(n_1+1)},\ldots,x_k^{(1)},\ldots,x_k^{(n_k+1)})^t.
\]
Here \(f^{[\delta]}\) is the function given by Fact~\ref{Chain rule for definable functions}. Thus \(\mathcal S(f)\) is obtained by applying the \(T\)-derivation chain rule to \(f\) and replacing each \(\delta(x_i^{(j)})\) by the new variable \(x_i^{(j+1)}\).

\subsection{One variable configurations}

We first study \(\delta\)-cells in one variable. This is the base case for the higher-dimensional construction. Let \(C\subseteq M\) be a one-variable \(\delta\)-cell, and let \(C^{\mathcal L}\) be its minimal source cell when it is defined.

We shall repeatedly use the following elementary consequence of the density condition.

\begin{lemma}\label{Projection of source cell is a source cell}
Let \(l\in\mathbb N\cup\{\infty\}\), let \(C\subseteq M\) be \(\mathcal L^\delta(M)\)-definable, and let \(\varphi\) be an \(\mathcal L^\delta(M)\)-formula such that \(C_\varphi^{\mathcal L}\) is a \(\mathcal C^l\)-cell, \(\ord(\varphi)=m>0\), and \(\Jet_m(C)\) is dense in \(C_\varphi^{\mathcal L}\). Then \(\Pi_{m-1}(C_\varphi^{\mathcal L})\) is a \(\mathcal C^l\)-cell, and \(\Jet_{m-1}(C)\) is dense in \(\Pi_{m-1}(C_\varphi^{\mathcal L})\).
\end{lemma}

\begin{proof}
By the inductive definition of cells, \(\Pi_{m-1}(C_\varphi^{\mathcal L})\) is a \(\mathcal C^l\)-cell. It remains to prove density. Suppose not. Then there is a nonempty open set \(U\subseteq M^m\) such that \(U\cap\Pi_{m-1}(C_\varphi^{\mathcal L})\neq\varnothing\) and \(U\cap\Jet_{m-1}(C)=\varnothing\). Hence \((U\times M)\cap C_\varphi^{\mathcal L}\) is a nonempty open subset of \(C_\varphi^{\mathcal L}\). Since \(\Jet_m(C)\) is dense in \(C_\varphi^{\mathcal L}\), there is \(a\in C\) such that \(\Jet_m(a)\in (U\times M)\cap C_\varphi^{\mathcal L}\). Projecting to the first \(m\) coordinates gives \(\Jet_{m-1}(a)\in U\cap\Jet_{m-1}(C)\), a contradiction.
\end{proof}

We next record the basic going-up principle. It says that once a jet-level set is given as the graph of a smooth definable function, the next jet coordinate is forced by the \(T\)-derivation chain rule.

\begin{lemma}\label{Going up property}
Fix \(l>0\), let \(A\subseteq M\), and let \(Y\subseteq M^n\). Let \(f:Y\to M\) be an \(\mathcal L(A)\)-definable function of class \(\mathcal C^l\), and set \(X=\Gamma(f)\subseteq M^{n+1}\). If \(\Jet_n(M)\cap X\) is dense in \(X\), then there exists a unique \(\mathcal L(A)\)-definable function \(g:X\to M\) of class \(\mathcal C^{l-1}\) such that \(\Jet_{n+1}(\Jet_n^{-1}(X))\) is dense in \(\Gamma(g)\).
\end{lemma}

\begin{proof}
By assumption, there exists an \(\mathcal L(A)\)-definable open set \(U\supseteq Y\) and an \(\mathcal L(A)\)-definable \(\mathcal C^l\)-function \(F:U\to M\) such that \(F|_Y=f\). Consider the \(\mathcal L(A)\)-definable \(\mathcal C^{l-1}\)-function \(\mathcal S(F)\) on \(\mathcal S(U)\), and set \(g:=\mathcal S(F)|_X\). Then \(g:X\to M\) is \(\mathcal L(A)\)-definable of class \(\mathcal C^{l-1}\). If \((a,\delta a,\ldots,\delta^n a)\in X\), then by the definition of \(\mathcal S\), we have \(g(a,\delta a,\ldots,\delta^n a)=\delta(f(a,\delta a,\ldots,\delta^{n-1}a))=\delta^{n+1}a\). Hence \(\Jet_{n+1}(\Jet_n^{-1}(X))\subseteq \Gamma(g)\).

We show density. Let \(V\subseteq M^{n+2}\) be open and suppose \(V\cap\Gamma(g)\neq\varnothing\). Choose \(p\in X\) such that \((p,g(p))\in V\). Since \(g\) is continuous, there is an open neighbourhood \(W\) of \(p\) in \(X\) such that \(\{(q,g(q)):q\in W\}\subseteq V\). Since \(\Jet_n(M)\cap X\) is dense in \(X\), choose \(a\in M\) such that \(\Jet_n(a)\in W\). Then \((\Jet_n(a),g(\Jet_n(a)))=(a,\delta a,\ldots,\delta^{n+1}a)\) belongs to \(V\cap\Jet_{n+1}(\Jet_n^{-1}(X))\). Thus \(\Jet_{n+1}(\Jet_n^{-1}(X))\) is dense in \(\Gamma(g)\).

It remains to show uniqueness. Suppose that \(h:X\to M\) is another \(\mathcal L(A)\)-definable \(\mathcal C^{l-1}\)-function such that \(\Jet_{n+1}(\Jet_n^{-1}(X))\) is dense in \(\Gamma(h)\). Since the same set is contained in both \(\Gamma(g)\) and \(\Gamma(h)\), and both graphs are closed in \(X\times M\), we have \(\Gamma(g)=\Gamma(h)\). Hence \(g=h\).
\end{proof}

\begin{remark}
By uniqueness, we may write \(\mathcal S(f):=g\), where \(g\) is the function obtained in Lemma~\ref{Going up property}. If the domain of \(f\) is open, this agrees with the earlier definition of \(\mathcal S(f)\).
\end{remark}

The previous lemma can be iterated. Starting from one graph layer, the higher graph layers are obtained by repeatedly applying the operator \(\mathcal S\).

\begin{corollary}\label{Iterated going up}
Fix \(l>0\) and let \(A\subseteq M\). Let \(f:Y\to M\) be an \(\mathcal L(A)\)-definable function of class \(\mathcal C^l\), where \(Y\subseteq M^n\), and set \(X=\Gamma(f)\). If \(\Jet_n(M)\cap X\) is dense in \(X\), then there exists a unique sequence of \(\mathcal L(A)\)-definable functions
\[
g_j:\Gamma(g_{j-1})\to M
\qquad (j=1,\ldots,l),
\]
with \(g_0=f\), such that, for each \(j=1,\ldots,l\), the function \(g_j\) is of class \(\mathcal C^{l-j}\) and
\[
\Jet_{n+j}\bigl(\Jet_{n+j-1}^{-1}(\Gamma(g_{j-1}))\bigr)
\]
is dense in \(\Gamma(g_j)\).
\end{corollary}

\begin{proof}
Existence and uniqueness follow by applying Lemma~\ref{Going up property} repeatedly.
\end{proof}

We also need the following uniqueness consequence. It says that if a higher graph layer has the correct dense set of realized jets, then its defining function is forced.

\begin{corollary}\label{Going up graph uniqueness}
Fix \(l>0\) and let \(A\subseteq M\). Let \(f:Y\to M\) be an \(\mathcal L(A)\)-definable function of class \(\mathcal C^l\), where \(Y\subseteq M^n\), and set \(X=\Gamma(f)\). Suppose that \(\Jet_n(M)\cap X\) is dense in \(X\). Let \(h:X\to M\) be an \(\mathcal L(A)\)-definable function of class \(\mathcal C^{l-1}\). If \(\Jet_{n+1}(\Jet_n^{-1}(X))\) is dense in \(\Gamma(h)\), then \(h=\mathcal S(f)\).
\end{corollary}

\begin{proof}
By Lemma~\ref{Going up property}, \(\Jet_{n+1}(\Jet_n^{-1}(X))\) is dense in \(\Gamma(\mathcal S(f))\). By assumption, the same set is dense in \(\Gamma(h)\). Since both \(\Gamma(h)\) and \(\Gamma(\mathcal S(f))\) are graphs of continuous functions over \(X\), they are closed in \(X\times M\). Hence their closures in \(X\times M\) are equal, and so \(\Gamma(h)=\Gamma(\mathcal S(f))\). Therefore \(h=\mathcal S(f)\).
\end{proof}

This uniqueness is the mechanism which detects redundant derivative coordinates. If a source cell becomes a graph before its last coordinate, then all later graph layers are forced by the chain rule. In that case the final derivative coordinate contributes no new information.

\begin{lemma}\label{Graph layer lowers order}
Let \(C\subseteq M\) be a \(\delta\)-cell, and let \(\varphi\) be a source formula for \(C\) of order \(m>0\). Suppose that, for some \(0\leq j<m\), the cell \(\Pi_j(C_\varphi^{\mathcal L})\) is the graph of an \(\mathcal L(M)\)-definable function over \(\Pi_{j-1}(C_\varphi^{\mathcal L})\), where for \(j=0\) this means that \(\Pi_0(C_\varphi^{\mathcal L})\) is a point. Then \(C\) has a source formula of order at most \(m-1\).
\end{lemma}

\begin{proof}
Let \(Y_i:=\Pi_i(C_\varphi^{\mathcal L})\) for \(i=0,\ldots,m\). By assumption, \(Y_j\) is a graph over \(Y_{j-1}\), with the convention that \(Y_{-1}\) is a point when \(j=0\). Since \(\Jet_m(C)\) is dense in \(Y_m\), repeated application of Lemma~\ref{Projection of source cell is a source cell} shows that \(\Jet_i(C)\) is dense in \(Y_i\) for every \(i=0,\ldots,m\).

We claim that every layer after \(Y_j\) is also a graph. Suppose \(Y_i=\Gamma(f_i)\) for some \(j\leq i<m\). Since \(\Jet_i(C)\) is dense in \(Y_i\), Lemma~\ref{Going up property} gives a graph \(\Gamma(\mathcal S(f_i))\) in which \(\Jet_{i+1}(C)\) is dense. But \(\Jet_{i+1}(C)\) is also dense in \(Y_{i+1}\). Hence \(Y_{i+1}\) cannot be a band over \(Y_i\), since a graph cannot be dense in a band. Therefore \(Y_{i+1}\) is a graph, and by Corollary~\ref{Going up graph uniqueness} it is precisely \(\Gamma(\mathcal S(f_i))\).

By induction, every layer \(Y_i\) for \(j\leq i\leq m\) is a graph, and each defining function is obtained from the previous one by applying \(\mathcal S\). In particular, the last coordinate of \(Y_m\) is determined by \(Y_{m-1}\) on realized jets. Hence membership in \(C\) is already determined by \(\Jet_{m-1}(x)\).

Let \(\theta(x^{(0)},\ldots,x^{(m-1)})\) be an \(\mathcal L(M)\)-formula defining \(Y_{m-1}\), and let \(\theta^*(x)\) be the corresponding \(\mathcal L^\delta(M)\)-formula obtained by substituting \((x,\delta x,\ldots,\delta^{m-1}x)\). Then \(\theta^*\) defines \(C\). Moreover, \(Y_{m-1}\) is a cell, and \(\Jet_{m-1}(C)\) is dense in \(Y_{m-1}\). Hence \(\theta^*\) is a source formula for \(C\) of order at most \(m-1\).
\end{proof}

We now apply the preceding reduction to source formulas of least possible order.

\begin{proposition}\label{One variable source cell shape}
Let \(C\subseteq M\) be a \(\delta\)-cell, and let \(\varphi\) be a minimal source formula for \(C\) of order \(m\), with \(C_\varphi^{\mathcal L}\) of class \(\mathcal C^l\). If \(m=0\), then \(C\) is either a point or an open interval. If \(m>0\), then the source cell \(C_\varphi^{\mathcal L}\) is either an open \(\mathcal C^l\)-cell in \(M^{m+1}\), or the graph of a \(\mathcal C^l\) \(\mathcal L(M)\)-definable function \(f:\Pi_{m-1}(C_\varphi^{\mathcal L})\to M\), where \(\Pi_{m-1}(C_\varphi^{\mathcal L})\) is an open \(\mathcal C^l\)-cell in \(M^m\).
\end{proposition}

\begin{proof}
The case \(m=0\) is immediate, since a cell in \(M\) is either a point or an open interval. Suppose \(m>0\). For \(i=0,\ldots,m\), set \(Y_i:=\Pi_i(C_\varphi^{\mathcal L})\), so that \(Y_m=C_\varphi^{\mathcal L}\). By the inductive definition of o-minimal cells, \(Y_0\) is either a point or an open interval, and for each \(i=1,\ldots,m\), the cell \(Y_i\) is either a graph over \(Y_{i-1}\) or a band between two continuous definable functions over \(Y_{i-1}\).

We claim that \(Y_0\) is not a point and that \(Y_i\) is not a graph over \(Y_{i-1}\) for any \(1\leq i<m\). If either occurred, Lemma~\ref{Graph layer lowers order} would produce a source formula for \(C\) of order at most \(m-1\), contradicting the minimality of \(\varphi\). Hence \(Y_0\) is an open interval, and for every \(1\leq i<m\), the layer \(Y_i\) is a band over \(Y_{i-1}\). Therefore \(Y_{m-1}\) is an open cell in \(M^m\).

It remains to consider the final layer \(Y_m=C_\varphi^{\mathcal L}\) over \(Y_{m-1}\). If \(Y_m\) is a band over \(Y_{m-1}\), then \(C_\varphi^{\mathcal L}\) is open in \(M^{m+1}\). If \(Y_m\) is a graph over \(Y_{m-1}\), then there is a \(\mathcal C^l\) \(\mathcal L(M)\)-definable function \(f:Y_{m-1}\to M\) such that \(C_\varphi^{\mathcal L}=\Gamma(f)\). This gives the desired alternatives.
\end{proof}

We can now make the minimal source cell canonical in one variable.

\begin{corollary}\label{Minimal source formula in one variable}
Let \(l\in\mathbb N\cup\{\infty\}\) and let \(A\subseteq M\). Suppose that \(C\subseteq M\) is an \(\mathcal L^\delta(A)\)-definable \(\mathcal C^l\) \(\delta\)-cell. Then there exists an \(\mathcal L^\delta(A)\)-source formula \(\psi(x)\) for \(C\) which is minimal. Moreover, if \(n=\ord(\psi)\), then the source cell \(C_\psi^{\mathcal L}\) is uniquely determined by \(C\) and \(n\).
\end{corollary}

\begin{proof}
By definition of \(\delta\)-cell, \(C\) has at least one source formula. Since the orders of one-variable source formulas are natural numbers, there is a source formula \(\psi\) of least possible order. Thus \(\psi\) is minimal. If \(\theta\) is another source formula for \(C\) of order \(n=\ord(\psi)\), then Lemma~\ref{Source cell is unique in one variable} gives \(C_\psi^{\mathcal L}=C_\theta^{\mathcal L}\). Hence the source cell of minimal order is uniquely determined by \(C\) and \(n\).
\end{proof}

Thus, for a one-variable \(\delta\)-cell \(C\), we may write \(C^{\mathcal L}\) for its source cell of minimal order.

\begin{definition}
Let \(C\subseteq M\) be a one-variable \(\delta\)-cell. We say that \(C\) is of \emph{class \(\mathcal C^l\)} if its minimal source cell \(C^{\mathcal L}\) is of class \(\mathcal C^l\). Similarly, we say that \(C\) is \emph{regular} or \emph{smooth} if \(C^{\mathcal L}\) is regular or smooth, respectively.
\end{definition}

\begin{remark}
This terminology refers to the smoothness of the associated source cell, not to ordinary Euclidean smoothness of the underlying subset of \(M\).
\end{remark}

The minimal source cell records the geometry of \(C\) at the jet level \(\ord(C)\). Its lower coordinate projections record the compatible lower-order jet data. We package these projections as the configuration of \(C\).

\begin{definition}
Let \(C\subseteq M\) be a one-variable \(\delta\)-cell, and let \(n\) be the order of a minimal source formula for \(C\). The \emph{configuration} of \(C\) up to order \(n\) is the sequence \((C_j)_{j=0}^n\) defined by
\[
C_j:=\Pi_j(C^{\mathcal L})
\]
for each \(j=0,\ldots,n\). We call \(C_j\) the \emph{\(j^{\mathrm{th}}\) configuration} of \(C\).
\end{definition}

By construction, \(\Jet_j(C)\) is dense in \(C_j\) for every \(j=0,\ldots,n\). We next extend this configuration beyond the order of \(C\). The point is that, once the minimal source cell is a graph, the higher jet coordinates are obtained by repeatedly applying \(\mathcal S\).

\begin{definition}\label{Iterated S operator}
Let \(A\subseteq M\), and let \(f:X\to M\) be an \(\mathcal L(A)\)-definable function of class \(\mathcal C^l\). Define \(\mathcal S^{[0]}(f):=f\). Having defined \(\mathcal S^{[i]}(f)\), we regard it as a function on its graph domain, and define \(\mathcal S^{[i+1]}(f)\) to be the function obtained from Lemma~\ref{Going up property} applied to \(\mathcal S^{[i]}(f)\). Equivalently, \(\mathcal S^{[i+1]}(f)\) is the restriction of \(\mathcal S(\mathcal S^{[i]}(f))\) to \(\Gamma(\mathcal S^{[i]}(f))\). Thus \(\mathcal S^{[i]}(f)\) is defined for \(0\leq i\leq l\), whenever the successive graph-lifting domains are understood.
\end{definition}

The iterated operator describes the higher configurations of a one-variable \(\delta\)-cell. The open case is controlled directly by genericity, while the graph case is controlled by repeated applications of Lemma~\ref{Going up property}.

\begin{proposition}\label{Complete configuration going up}
Let \(l\in\mathbb N\cup\{\infty\}\), and let \(C\subseteq M\) be a \(\mathcal C^l\) \(\delta\)-cell of order \(n\). Then, for each \(r\) with \(0\leq r\leq l\), there is an \(\mathcal L(M)\)-definable o-minimal cell \(C_{n+r}\subseteq M^{n+r+1}\) such that \(\Jet_{n+r}(C)\) is dense in \(C_{n+r}\). More precisely:
\begin{itemize}
\item[(i)] if \(C\) is a singleton, then \(C_{n+r}=\Jet_{n+r}(C)\);
\item[(ii)] if \(C^{\mathcal L}\) is open in \(M^{n+1}\), then \(C_{n+r}=C^{\mathcal L}\times M^r\);
\item[(iii)] if \(n>0\) and \(C^{\mathcal L}=\Gamma(f)\), where \(f:\Pi_{n-1}(C^{\mathcal L})\to M\) is an \(\mathcal L(M)\)-definable function of class \(\mathcal C^l\), then \(C_{n+r}=\Gamma(\mathcal S^{[r]}(f))\).
\end{itemize}
\end{proposition}

\begin{proof}
The singleton case is immediate. Suppose next that \(C^{\mathcal L}\) is open in \(M^{n+1}\). Then \(C=\Jet_n^{-1}(C^{\mathcal L})\), and
\[
\Jet_{n+r}(C)
=
\Jet_{n+r}(M)\cap (C^{\mathcal L}\times M^r).
\]
Since \(\Jet_{n+r}(M)\) is dense in \(M^{n+r+1}\) by Fact~\ref{Jet space is dense}, it follows that \(\Jet_{n+r}(C)\) is dense in \(C^{\mathcal L}\times M^r\).

It remains to consider the graph case. Suppose that \(C^{\mathcal L}=\Gamma(f)\), where \(f:\Pi_{n-1}(C^{\mathcal L})\to M\) is of class \(\mathcal C^l\). Since \(\Jet_n(C)\) is dense in \(C^{\mathcal L}\), Lemma~\ref{Going up property} gives a unique function \(\mathcal S(f)\) such that \(\Jet_{n+1}(C)\) is dense in \(\Gamma(\mathcal S(f))\). Applying the same lemma repeatedly gives, for each \(r\leq l\), that \(\Jet_{n+r}(C)\) is dense in \(\Gamma(\mathcal S^{[r]}(f))\). Thus \(C_{n+r}=\Gamma(\mathcal S^{[r]}(f))\), as required.
\end{proof}

We now extend the configuration of \(C\) beyond its order.

\begin{definition}\label{Complete configuration of one variable cell}
Let \(l\in\mathbb N\cup\{\infty\}\), and let \(C\subseteq M\) be a \(\mathcal C^l\) one-variable \(\delta\)-cell of order \(n\). The \emph{complete configuration} of \(C\) is the sequence \((C_j)\) defined as follows. For \(0\leq j\leq n\), set \(C_j=\Pi_j(C^{\mathcal L})\). For \(j=n+r\) with \(0\leq r\leq l\), define \(C_j\) as in Proposition~\ref{Complete configuration going up}. When \(l=\infty\), this gives a sequence indexed by all \(j\geq0\).
\end{definition}

For every \(j\) in the range of the complete configuration, \(\Jet_j(C)\) is dense in \(C_j\). Once the minimal source cell \(C^{\mathcal L}\) is fixed, the terms with \(j\leq n\) are its projections. The terms with \(j>n\) are the cells given by Proposition~\ref{Complete configuration going up}; in the graph case, these higher terms are uniquely determined by repeated applications of Lemma~\ref{Going up graph uniqueness}.

The preceding results show that a one-variable \(\delta\)-cell has a very simple configuration: it is open up to its minimal graph level, and after that all higher levels are forced by the chain rule.

\begin{corollary}\label{Layout of one dimensional delta cell}
Let \(l\in\mathbb N\cup\{\infty\}\), and let \(C\subseteq M\) be a \(\mathcal C^l\) one-variable \(\delta\)-cell. Let \((C_j)\) be its complete configuration. Then there is a unique \(n_0\in\mathbb N\cup\{\infty\}\) such that \(C_j\) is open in \(M^{j+1}\) for all \(j<n_0\), and, if \(n_0<\infty\), \(C_{n_0}\) is a graph over \(C_{n_0-1}\). In this case, \(C_j\) is a graph over \(C_{j-1}\) for every \(j\geq n_0\) in the range of the complete configuration.
\end{corollary}

\begin{proof}
By Proposition~\ref{One variable source cell shape}, the minimal source cell \(C^{\mathcal L}\) is either open, or it is the first graph layer over an open cell. In the first case, Proposition~\ref{Complete configuration going up} gives that every higher configuration is open, so we take \(n_0=\infty\). In the second case, let \(n_0\) be the first level at which the minimal source cell is a graph. Then \(C_j\) is open for \(j<n_0\), and \(C_{n_0}\) is a graph over \(C_{n_0-1}\). Proposition~\ref{Complete configuration going up} then implies that every higher configuration is obtained by repeated applications of Lemma~\ref{Going up property}, and hence is a graph over the preceding configuration. Uniqueness of \(n_0\) is immediate from the defining condition.
\end{proof}

\subsection{The general case}

We now pass from one variable to several variables. The main new issue is that a \(\delta\)-cell may admit many source formulas. Some source formulas may introduce redundant higher derivative coordinates, and these extra coordinates can obscure the fiber behavior of the associated source cell.

For example, consider
\[
C=\{(x,y)\in M^2:\delta y=x\}.
\]
The formula \(\delta y=x\) gives a natural source cell in coordinates \((x,y,\delta y)\). However, the equivalent formula
\[
\delta y=x\wedge \delta^2y=\delta^2y
\]
has source set in coordinates \((x,y,\delta y,\delta^2y)\) given by
\[
D=\{(x,y,y',y'')\in M^4:y'=x\}.
\]
The realized jets are dense in \(D\), so this is still a valid source cell for \(C\). But after taking the fiber over \(x=a\), the source fiber is
\[
D(a)=\{(y,y',y'')\in M^3:y'=a\},
\]
whereas the realized jets in the actual fiber \(C_a=\{y:\delta y=a\}\) are
\[
\{(y,a,\delta a):y\in M\}.
\]
This set is not dense in \(D(a)\). Thus arbitrary source formulas may introduce extra derivative coordinates which destroy the expected fiber behavior.

There is also a blockwise issue. For instance, the set
\[
C=\{(x,y)\in M^2:\delta y=x\wedge \delta^2y>0\}
\]
can be defined by a formula of order \((0,2)\). It can also be defined by
\[
\delta y=x\wedge \delta x>0,
\]
which has order \((1,1)\). The second presentation distributes the derivative coordinates more naturally between the two blocks. This suggests that, in several variables, one should compare source formulas by allowing derivative order to move from later blocks to earlier blocks.

We order \(\mathbb N^k\) by reverse lexicographic order: for \(m,n\in\mathbb N^k\), we write \(m\prec_{\mathrm{rlex}} n\) if, letting \(i\) be the largest index such that \(m_i\neq n_i\), we have \(m_i<n_i\).

Let \(C\subseteq M^k\) be a \(\delta\)-cell, and let \(\varphi\) be a source formula for \(C\) with \(\ord(\varphi)=n=(n_1,\ldots,n_k)\). Write \(D:=C_\varphi^{\mathcal L}\subseteq E_n\). We group the coordinates of \(E_n\) into \(k\) blocks, where the \(i^{\mathrm{th}}\) block is
\[
x_i^{(0)},\ldots,x_i^{(n_i)}.
\]
For \(1\leq i\leq k\) and \(0\leq j\leq n_i+1\), let
\[
\rho_{i,j}:E_n\to M^{(n_1+1)+\cdots+(n_{i-1}+1)+j}
\]
be the coordinate projection keeping all coordinates in the first \(i-1\) blocks and the first \(j\) coordinates in the \(i^{\mathrm{th}}\) block. Thus \(\rho_{i,0}\) keeps only the previous blocks, while \(\rho_{i,n_i+1}\) keeps the first \(i\) full blocks. This notation is separate from the jet-truncation maps \(\Pi_m\); the maps \(\rho_{i,j}\) are used only to describe the inductive construction of the source cell one coordinate at a time.

\begin{definition}\label{Definition of Layout}
A \emph{layout} of \(C\) with respect to the source formula \(\varphi\) is the \(k\)-tuple of vectors of pairs of functions
\[
\mathcal F_C^\varphi=
\begin{pmatrix}
\begin{pmatrix}
f_{1,0,0},f_{1,0,1}\\
\cdots\\
f_{1,n_1,0},f_{1,n_1,1}
\end{pmatrix},
\ldots,
\begin{pmatrix}
f_{k,0,0},f_{k,0,1}\\
\cdots\\
f_{k,n_k,0},f_{k,n_k,1}
\end{pmatrix}
\end{pmatrix}
\]
obtained from a chosen inductive cell presentation of \(D=C_\varphi^{\mathcal L}\) as follows. For each \(1\leq i\leq k\) and \(0\leq j\leq n_i\), the cell \(\rho_{i,j+1}(D)\) is either the graph of an \(\mathcal L(M)\)-definable function \(f_{i,j,0}:\rho_{i,j}(D)\to M\), in which case we set \(f_{i,j,1}=f_{i,j,0}\), or the band \((f_{i,j,0},f_{i,j,1})_{\rho_{i,j}(D)}\) between two extended \(\mathcal L(M)\)-definable boundary functions \(f_{i,j,0}<f_{i,j,1}\) on \(\rho_{i,j}(D)\), where we allow \(f_{i,j,0}=-\infty\) and \(f_{i,j,1}=+\infty\).

The \(i^{\mathrm{th}}\) vector is called the \emph{\(i^{\mathrm{th}}\) block} of the layout, and the pair \((f_{i,j,0},f_{i,j,1})\) is called the \emph{\(j^{\mathrm{th}}\) layer} of the \(i^{\mathrm{th}}\) block.
\end{definition}

\begin{definition}\label{Definition of prepared block and layout}
Using the notation in Definition~\ref{Definition of Layout}, the \(i^{\mathrm{th}}\) block of a layout \(\mathcal F_C^\varphi\) is called \emph{prepared} if there exists \(0\leq n_i'\leq n_i+1\) such that \(f_{i,j,0}<f_{i,j,1}\) for all \(0\leq j<n_i'\), and \(f_{i,j,0}=f_{i,j,1}\) for all \(n_i'\leq j\leq n_i\). The layout \(\mathcal F_C^\varphi\) is called \emph{prepared} if each of its blocks is prepared. A source formula is called \emph{prepared} if its layout is prepared.
\end{definition}

In the following arguments, we begin with source cells of class \(\mathcal C^r\), for \(r\) sufficiently large. This can be arranged uniformly for the finite decompositions used below. Indeed, fix an initial source order \(n=(n_1,\ldots,n_k)\). The order-reduction procedure below is finite: each step lowers the source order with respect to \(\prec_{\mathrm{rlex}}\), and although it may add finitely many derivative coordinates to earlier blocks, the number of such additions is bounded in terms of the current order. Hence, starting from \(n\), only finitely many applications of the operator \(\mathcal S\) are needed before a minimal prepared layout is reached. Let \(r(n)\) be large enough for all these applications. Since a finite refined \(\delta\)-decomposition involves only finitely many source orders, we may choose \(r\) larger than all the corresponding \(r(n)\). By Theorem~\ref{Refined delta decomposition theorem}, the source cells may be chosen of class \(\mathcal C^r\), and therefore all minimal source formulas obtained in the reduction have the differentiability required below.

\begin{definition}\label{Definition of minimal source formula and layout}
Let \(C\subseteq M^k\) be a \(\delta\)-cell. A source formula \(\varphi\) for \(C\) is called \emph{minimal} if \(\ord(\varphi)\) is minimal, among the orders of all source formulas for \(C\), with respect to \(\prec_{\mathrm{rlex}}\). A layout arising from a minimal source formula is called a \emph{minimal layout} of \(C\).
\end{definition}

Since \(\prec_{\mathrm{rlex}}\) is well-founded on \(\mathbb N^k\), every \(\delta\)-cell admits a minimal source formula, and hence admits a minimal layout. We next show that, whenever the chosen minimal source formula is sufficiently smooth for the required applications of \(\mathcal S\), its layout is prepared.

We prove this in three steps. First, if differentiating a graph layer produces a function depending only on the coordinates already present at the next level, then the next layer is forced to be a graph.

\begin{lemma}\label{Factorization forces graph layer}
Let \(C\subseteq M^k\) be a \(\delta\)-cell, let \(\varphi\) be a source formula for \(C\), and write \(D=C_\varphi^{\mathcal L}\). Suppose that, in the \(i^{\mathrm{th}}\) block of the layout of \(D\), the \(j^{\mathrm{th}}\) layer is the graph of an \(\mathcal L(M)\)-definable \(\mathcal C^1\)-function \(x_i^{(j)}=g_j\), where \(g_j:\rho_{i,j}(D)\to M\). Suppose moreover that \(\mathcal S(g_j)\) factors through \(\rho_{i,j+1}(D)\); that is, there is an \(\mathcal L(M)\)-definable function \(\widetilde g_{j+1}:\rho_{i,j+1}(D)\to M\) such that \(\mathcal S(g_j)=\widetilde g_{j+1}\circ\rho_{i,j+1}\) on the lifted domain above \(\rho_{i,j+1}(D)\). Then the \((j+1)^{\mathrm{st}}\) layer is the graph of \(\widetilde g_{j+1}\).
\end{lemma}

\begin{proof}
On realized jets, differentiating the identity \(x_i^{(j)}=g_j\) gives
\[
x_i^{(j+1)}=\mathcal S(g_j)=\widetilde g_{j+1}(\rho_{i,j+1}).
\]
If the \((j+1)^{\mathrm{st}}\) layer is already a graph, say the graph of \(h:\rho_{i,j+1}(D)\to M\), then \(h=\widetilde g_{j+1}\) on the dense subset \(\rho_{i,j+1}(\Jet_n(C))\) of \(\rho_{i,j+1}(D)\). By continuity, \(h=\widetilde g_{j+1}\) on all of \(\rho_{i,j+1}(D)\).

It remains to rule out the band case. Suppose that the \((j+1)^{\mathrm{st}}\) layer is a band over \(\rho_{i,j+1}(D)\). Choose \(p\in\rho_{i,j+1}(D)\). The fiber of this band over \(p\) is a nonempty open interval, so it contains an open subinterval disjoint from \(\{\widetilde g_{j+1}(p)\}\). By continuity of the boundary functions and of \(\widetilde g_{j+1}\), there is a nonempty relatively open subset \(U\) of \(D\) on which \(x_i^{(j+1)}\neq \widetilde g_{j+1}(\rho_{i,j+1})\). Since \(\Jet_n(C)\) is dense in \(D\), the set \(U\) contains a realized jet. This contradicts the displayed identity on realized jets. Hence the \((j+1)^{\mathrm{st}}\) layer cannot be a band, and therefore it is the graph of \(\widetilde g_{j+1}\).
\end{proof}

The second step identifies where the first obstruction to preparedness must come from. If a block has a graph layer followed later by a band layer, then the formal derivative of the last graph layer before the band must involve a derivative coordinate from an earlier block which is not yet present.

\begin{lemma}\label{First band detects missing earlier derivative}
Let \(C\subseteq M^k\) be a \(\delta\)-cell, let \(\varphi\) be a sufficiently smooth source formula for \(C\), and write \(\ord(\varphi)=n=(n_1,\ldots,n_k)\) and \(D=C_\varphi^{\mathcal L}\). Let \(T(\mathcal F_C^\varphi)=(t_{i,j})\) be the type of the layout. Suppose that, for some block \(i\), there are indices \(0\leq j<s\leq n_i\) such that \(t_{i,j}=0\) and \(t_{i,s}=1\). Let \(j_0\) be the least index \(j\) for which there exists \(s>j\) with \(t_{i,j}=0\) and \(t_{i,s}=1\), and let \(s_0>j_0\) be the least index such that \(t_{i,s_0}=1\). For \(j_0\leq j<s_0\), write the corresponding graph layer as \(x_i^{(j)}=g_j\), where \(g_j:\rho_{i,j}(D)\to M\) is \(\mathcal L(M)\)-definable. Then the formal derivative \(\mathcal S(g_{s_0-1})\) does not factor through \(\rho_{i,s_0}(D)\). Moreover, if a jet coordinate occurring in the formal expression for \(\mathcal S(g_{s_0-1})\) is not kept by \(\rho_{i,s_0}\) and is not the output coordinate \(x_i^{(s_0)}\), then it is of the form \(x_h^{(n_h+1)}\) for some earlier block \(h<i\).
\end{lemma}

\begin{proof}
If \(\mathcal S(g_{s_0-1})\) factored through \(\rho_{i,s_0}(D)\), then there would be an \(\mathcal L(M)\)-definable function \(\widetilde g_{s_0}:\rho_{i,s_0}(D)\to M\) such that \(\mathcal S(g_{s_0-1})=\widetilde g_{s_0}\circ\rho_{i,s_0}\) on the lifted domain above \(\rho_{i,s_0}(D)\). By Lemma~\ref{Factorization forces graph layer}, applied with \(j=s_0-1\), the \(s_0^{\mathrm{th}}\) layer would then be the graph of \(\widetilde g_{s_0}\). This contradicts the choice of \(s_0\) as a band layer. Hence \(\mathcal S(g_{s_0-1})\) does not factor through \(\rho_{i,s_0}(D)\).

It remains to identify where the failure of factorization can come from. Since \(g_{s_0-1}\) is defined on \(\rho_{i,s_0}(D)\), the coordinates on which it depends are among
\[
\{x_h^{(r)}:h<i,\ 0\leq r\leq n_h\}\cup\{x_i^{(r)}:0\leq r\leq s_0-1\}.
\]
By the definition of \(\mathcal S\), every coordinate occurring in the formal expression for \(\mathcal S(g_{s_0-1})\) is either already among these coordinates or is the formal derivative \(x_h^{(r+1)}\) of one of them. The coordinates \(x_i^{(0)},\ldots,x_i^{(s_0-1)}\) are kept by \(\rho_{i,s_0}\), and \(x_i^{(s_0)}\) is the output coordinate of the next layer. Therefore any coordinate occurring in \(\mathcal S(g_{s_0-1})\) which is not kept by \(\rho_{i,s_0}\), apart from \(x_i^{(s_0)}\), must be the formal derivative of an earlier-block coordinate of maximal available order. Hence it is of the form \(x_h^{(n_h+1)}\) for some \(h<i\).
\end{proof}

The third step replaces the last block which is not prepared by moving the missing derivative coordinates to earlier blocks. This produces a new source formula of smaller reverse lexicographic order.

\begin{lemma}\label{Last bad block replacement}
Let \(C\subseteq M^k\) be a \(\delta\)-cell, let \(\varphi\) be a sufficiently smooth source formula for \(C\), and write \(\ord(\varphi)=n=(n_1,\ldots,n_k)\). Let \(\mathcal F_C^\varphi\) be the layout of \(C\) with respect to \(\varphi\), and let \(T(\mathcal F_C^\varphi)=(t_{h,j})\) be its type. Suppose that \(i\) is the largest block index for which the \(i^{\mathrm{th}}\) block is not prepared. Then there is another source formula \(\psi\) for \(C\), with order \(m=(m_1,\ldots,m_k)\), such that \(m_h=n_h\) for all \(h>i\), and \(m_i<n_i\). In particular, \(m\prec_{\mathrm{rlex}} n\).
\end{lemma}

\begin{proof}
Let \(D:=C_\varphi^{\mathcal L}\). Since the \(i^{\mathrm{th}}\) block is not prepared, there are \(0\leq j<s\leq n_i\) such that \(t_{i,j}=0\) and \(t_{i,s}=1\). Let \(j_0\) and \(s_0\) be as in Lemma~\ref{First band detects missing earlier derivative}. Thus \(t_{i,j}=0\) for \(j_0\leq j<s_0\), and \(t_{i,s_0}=1\). For \(j_0\leq j<s_0\), write the corresponding graph layer as \(x_i^{(j)}=g_j\), where \(g_j:\rho_{i,j}(D)\to M\) is \(\mathcal L(M)\)-definable.

Set \(q:=n_i-s_0+1\), and define \(m=(m_1,\ldots,m_k)\) by
\[
m_h=n_h+q\quad(h<i),\qquad m_i=s_0-1,\qquad m_h=n_h\quad(h>i).
\]

We first define the functions which recover the removed coordinates in the \(i^{\mathrm{th}}\) block. Let
\[
G_{s_0}:=\mathcal S(g_{s_0-1}).
\]
By Lemma~\ref{First band detects missing earlier derivative}, every additional jet coordinate occurring in \(G_{s_0}\), apart from the output coordinate \(x_i^{(s_0)}\), is of the form \(x_h^{(n_h+1)}\) for some \(h<i\). Hence \(G_{s_0}\) is defined after enlarging the earlier blocks by one coordinate. Suppose that \(G_{s_0},\ldots,G_r\) have been defined for some \(r<n_i\). Define
\[
G_{r+1}:=\mathcal S(G_r)\big|_{x_i^{(s_0)}=G_{s_0},\,\ldots,\,x_i^{(r)}=G_r}.
\]
Each application of \(\mathcal S\) raises the required order in an earlier block by at most one. Since \(r-s_0+2\leq q\), all coordinates occurring in \(G_{r+1}\) belong to the earlier blocks of order at most \(n_h+q\), together with \(x_i^{(0)},\ldots,x_i^{(s_0-1)}\). By construction, for every \(a\in C\),
\[
\delta^r a_i=G_r(\Jet_m(a))
\]
for \(s_0\leq r\leq n_i\).

We now construct the new source cell. For each \(h<i\), extend the \(h^{\mathrm{th}}\) block from order \(n_h\) to order \(n_h+q\). At each new layer, use the graph prescribed by \(\mathcal S\) if the preceding layer is a graph, and use the band \((-\infty,+\infty)\) if the preceding layer is a band. Together with the truncation of the \(i^{\mathrm{th}}\) block at \(s_0-1\) and the unchanged later blocks, this gives an ambient cell \(\Omega\subseteq E_m\).

Define a map \(F:\Omega\to E_n\) by retaining all coordinates which are still present, and by setting \((F(z))_i^{(r)}:=G_r(z)\) for \(s_0\leq r\leq n_i\). Let
\[
D_\psi:=F^{-1}(D).
\]
By the inductive cell presentation of \(D\), the set \(D_\psi\) is an o-minimal cell: graph layers remain graph layers after substitution, and each band \((u_0,u_1)\) is pulled back to the band \((u_0\circ F,u_1\circ F)\).

For \(a\in C\), the identities above give \(F(\Jet_m(a))=\Jet_n(a)\), and hence \(\Jet_m(a)\in D_\psi\). Conversely, if \(\Jet_m(a)\in D_\psi\), then \(F(\Jet_m(a))\in D\), and the graph identities \(x_i^{(j)}=g_j\) for \(j_0\leq j<s_0\), together with the definitions of the functions \(G_r\), imply \(F(\Jet_m(a))=\Jet_n(a)\). Hence \(a\in C\). Therefore
\[
\Jet_m(C)=D_\psi\cap\Jet_m(M^k).
\]

It remains to prove density of \(\Jet_m(C)\) in \(D_\psi\). We prove this by induction through the coordinate construction of \(D_\psi\). If the next layer is the graph of a continuous definable function, density is preserved by continuity and the chain-rule identities on realized jets. If the next layer is a band, then every nonempty relatively open subset of the band has nonempty projection to the free jet coordinate, and genericity gives a realized jet inside it. Thus \(\Jet_m(C)\) is dense in \(D_\psi\).

Let \(\theta_\psi\) be an \(\mathcal L(M)\)-formula defining \(D_\psi\), and set \(\psi(x):=\theta_\psi(\Jet_m(x))\). Then \(\psi\) defines \(C\), and \(D_\psi\) is a source cell for \(\psi\). Thus \(\psi\) is a source formula for \(C\) of order \(m\). Since \(m_h=n_h\) for all \(h>i\), while \(m_i=s_0-1<n_i\), we have \(m\prec_{\mathrm{rlex}} n\).
\end{proof}

We can now conclude that minimal layouts are prepared.

\begin{theorem}\label{Minimal layouts are prepared}
Using the notation in Definition~\ref{Definition of Layout}, let \(C\subseteq M^k\) be a \(\delta\)-cell, and let \(\mathcal F_C^\varphi\) be a minimal layout arising from a sufficiently smooth source formula \(\varphi\). Then \(\mathcal F_C^\varphi\) is prepared.
\end{theorem}

\begin{proof}
Suppose, toward a contradiction, that \(\mathcal F_C^\varphi\) is not prepared. Then there is at least one block which is not prepared. Let \(i\) be the largest block index such that the \(i^{\mathrm{th}}\) block is not prepared. By Lemma~\ref{Last bad block replacement}, there is another source formula \(\psi\) for \(C\), with order \(m\), such that
\[
m\prec_{\mathrm{rlex}}\ord(\varphi).
\]
This contradicts the minimality of \(\varphi\). Therefore every block of \(\mathcal F_C^\varphi\) is prepared, so \(\mathcal F_C^\varphi\) is prepared.
\end{proof}

We next show that the source configuration obtained from a minimal prepared layout is intrinsic. Although a \(\delta\)-cell may admit many defining formulas, minimality removes the ambiguity at the level of source cells, and the prepared form then forces all higher configurations by the chain-rule construction.

\begin{theorem}\label{Uniqueness of minimal source configurations}
Let \(C\subseteq M^k\) be a sufficiently smooth \(\delta\)-cell, and let \(\varphi\) and \(\psi\) be minimal source formulas for \(C\). Write \(n=\ord(\varphi)=\ord(\psi)\). Then
\[
C_\varphi^{\mathcal L}=C_\psi^{\mathcal L}\subseteq E_n.
\]
Moreover, if the corresponding minimal layouts are prepared, then they determine the same lower configurations and the same extended configurations. Hence the complete configuration of \(C\) is independent of the chosen minimal prepared source layout.
\end{theorem}

\begin{proof}
Set \(D:=C_\varphi^{\mathcal L}\) and \(D':=C_\psi^{\mathcal L}\). Both \(D\) and \(D'\) are o-minimal cells in \(E_n\), and since \(\varphi\) and \(\psi\) define the same set \(C\), we have
\[
D\cap\Jet_n(M^k)=D'\cap\Jet_n(M^k)=\Jet_n(C).
\]
By the definition of source cell, \(\Jet_n(C)\) is dense in both \(D\) and \(D'\), so \(\cl(D)=\cl(D')\).

We prove \(D=D'\) by induction through the coordinate construction of the cells. Write the coordinates of \(E_n\) blockwise as
\[
x_1^{(0)},\ldots,x_1^{(n_1)},\ldots,x_k^{(0)},\ldots,x_k^{(n_k)}.
\]
For each block \(i\) and each \(0\leq j\leq n_i+1\), let \(\rho_{i,j}\) be as in Definition~\ref{Definition of Layout}. We show, in the lexicographic order induced by this coordinate ordering, that \(\rho_{i,j}(D)=\rho_{i,j}(D')\).

The initial projection is clear. Suppose that the equality has been proved for the previous projection, and set \(B:=\rho_{i,j}(D)=\rho_{i,j}(D')\). The next projections \(\rho_{i,j+1}(D)\) and \(\rho_{i,j+1}(D')\) are cells over \(B\). Since \(\cl(D)=\cl(D')\), their closures over the common base agree. If one were a graph and the other a band, then the fibers over points of \(B\) would have different dimensions, contradicting equality of closures. Thus they are either both graphs or both bands.

If both are graphs, say \(\Gamma(f)\) and \(\Gamma(g)\) over \(B\), then equality of closures gives \(f=g\) on \(B\), since \(f\) and \(g\) are continuous. If both are bands, say \((f_0,f_1)_B\) and \((g_0,g_1)_B\), then equality of closures gives \(f_0=g_0\) and \(f_1=g_1\) on \(B\), again by continuity of the boundary functions. Hence the next layers agree. This completes the induction, and therefore \(D=D'\).

It follows immediately that the lower configurations agree, since for every \(m\leq n\) they are given by the same projection \(C_m=\Pi_m(D)=\Pi_m(D')\). It remains to check the extended configurations. Since the prepared layouts have the same source cell, the same graph and band layers occur at the last layer of each block. In a band layer, the extension adds a free derivative coordinate, so the next extended layer is uniquely determined. In a graph layer, the next extended layer is determined by applying the operator \(\mathcal S\), and uniqueness follows from Corollary~\ref{Going up graph uniqueness}. Repeating this argument gives equality of all extended configurations for which they are defined.
\end{proof}

From now on, when working with sufficiently smooth \(\delta\)-cells, we choose a minimal layout and, by Theorem~\ref{Minimal layouts are prepared}, regard it as prepared.

\begin{definition}\label{Definition of source layout}
Let \(C\) be a sufficiently smooth \(\delta\)-cell. A \emph{source layout} of \(C\) is a minimal prepared layout of \(C\). By Theorem~\ref{Minimal layouts are prepared}, every minimal layout arising from a sufficiently smooth source formula is prepared. By Theorem~\ref{Uniqueness of minimal source configurations}, the corresponding source cell is independent of the chosen minimal source formula. We denote this canonical source cell by \(C^{\mathcal L}\).
\end{definition}

We now record the complete configuration determined by a source layout. Let \(C\subseteq M^k\) be a sufficiently smooth \(\delta\)-cell, let \(\mathcal F_C\) be a source layout of order \(n=(n_1,\ldots,n_k)\), and write
\[
\mathcal F_C=
\begin{pmatrix}
\begin{pmatrix}
f_{1,0,0},f_{1,0,1}\\
\cdots\\
f_{1,n_1,0},f_{1,n_1,1}
\end{pmatrix},
\ldots,
\begin{pmatrix}
f_{k,0,0},f_{k,0,1}\\
\cdots\\
f_{k,n_k,0},f_{k,n_k,1}
\end{pmatrix}
\end{pmatrix}
\]
for its block presentation, and
\[
T(\mathcal F_C)=(t_{i,j})_{1\leq i\leq k,\ 0\leq j\leq n_i}
\]
for its type.

We first define the lower configurations. If \(m=(m_1,\ldots,m_k)\) satisfies \(0\leq m_i\leq n_i\) for all \(i\), set
\[
C_m:=\Pi_m(C^{\mathcal L})\subseteq E_m.
\]
Since \(\Jet_n(C)\) is dense in \(C^{\mathcal L}\), it follows that \(\Jet_m(C)\) is dense in \(C_m\).

We now extend the configuration upward. Suppose first that \(C\) is of class \(\mathcal C^1\). For each block \(1\leq i\leq k\), define the next pair \((f_{i,n_i+1,0},f_{i,n_i+1,1})\) as follows. If the last layer of the \(i^{\mathrm{th}}\) block is a graph, that is, if \(t_{i,n_i}=0\), set
\[
f_{i,n_i+1,0}=f_{i,n_i+1,1}:=\mathcal S(f_{i,n_i,0}).
\]
If the last layer is a band, that is, if \(t_{i,n_i}=1\), set
\[
f_{i,n_i+1,0}:=-\infty,\qquad f_{i,n_i+1,1}:=+\infty.
\]
This gives a new layout, denoted \(\mathcal F_C^{n+1}\), in the jet space of order \(n+1:=(n_1+1,\ldots,n_k+1)\). On realized jets, the graph case follows from the chain rule, while the band case adds a free derivative coordinate. Hence \(\Jet_{n+1}(C)\) is dense in the o-minimal cell defined by \(\mathcal F_C^{n+1}\).

More generally, if \(C\) is of class \(\mathcal C^p\), then this construction may be iterated for \(0\leq r\leq p\). We write \(\mathcal F_C^n:=\mathcal F_C\), and define \(\mathcal F_C^{n+r+1}\) from \(\mathcal F_C^{n+r}\) by the same rule. Let \(C_{n+r}\) be the o-minimal cell defined by \(\mathcal F_C^{n+r}\). Then \(\Jet_{n+r}(C)\) is dense in \(C_{n+r}\) for each \(0\leq r\leq p\), where \(n+r:=(n_1+r,\ldots,n_k+r)\).

\begin{definition}\label{Complete configuration in several variables}
Using the notation above, let \(C\) be a \(\mathcal C^p\) \(\delta\)-cell, where \(p\in\mathbb N\cup\{\infty\}\), and let \(\mathcal F_C\) be a source layout of order \(n=(n_1,\ldots,n_k)\). For each multi-index \(m=(m_1,\ldots,m_k)\) with \(0\leq m_i\leq n_i\) for all \(i\), the cell \(C_m:=\Pi_m(C^{\mathcal L})\) is called the \emph{\(m^{\mathrm{th}}\) lower configuration} of \(C\). For each \(0\leq r\leq p\), where \(n+r=(n_1+r,\ldots,n_k+r)\), the layout \(\mathcal F_C^{n+r}\) is called the \emph{\((n+r)^{\mathrm{th}}\) extended layout} of \(C\), and the cell \(C_{n+r}\) is called the \emph{\((n+r)^{\mathrm{th}}\) extended configuration} of \(C\). Together, the lower configurations \((C_m)_{m\leq n}\) and the extended configurations \((C_{n+r})_{0\leq r\leq p}\) are called the \emph{complete configuration} of \(C\).
\end{definition}

When \(k=1\), the lower and extended configurations defined above agree with the configurations introduced in the one-variable case.

We next record how source layouts behave under projections and fibers. The corresponding o-minimal cells are obtained by projecting the layout or by substituting a realized jet into its initial blocks. The point requiring proof is that the resulting cells still have dense realized jets.

\begin{theorem}\label{Projection and fibers of a delta cell is a delta cell}
Let \(C\subseteq M^k\) be a sufficiently smooth \(\delta\)-cell, and let \(\mathcal F_C\) be a source layout of order \(n=(n_1,\ldots,n_k)\), with canonical source cell \(C^{\mathcal L}\subseteq E_n\). Fix \(1\leq j<k\), and let \(a\in\Pi_j(C)\). Then:
\begin{itemize}
\item[(i)] the projection \(\Pi_j(C)\subseteq M^j\) is a sufficiently smooth \(\delta\)-cell, with source cell obtained by projecting \(C^{\mathcal L}\) onto its first \(j\) blocks;
\item[(ii)] the fiber \(C_a:=\{b\in M^{k-j}:(a,b)\in C\}\) is a sufficiently smooth \(\delta\)-cell, with source cell obtained by substituting \(\Jet_{(n_1,\ldots,n_j)}(a)\) into the last \(k-j\) blocks of the source layout \(\mathcal F_C\).
\end{itemize}
Moreover, the induced source layouts on the projection and the fiber are prepared.
\end{theorem}

\begin{proof}
The assertion for projections follows from the projection compatibility in Theorem~\ref{Refined delta decomposition theorem} and from the definition of the lower configurations. The source cell of \(\Pi_j(C)\) is the projection of \(C^{\mathcal L}\) onto its first \(j\) blocks, and the corresponding source layout is obtained by deleting the remaining blocks of \(\mathcal F_C\). Since \(\mathcal F_C\) is prepared, the projected layout is prepared.

We consider the fiber. Set
\[
D_a:=\{b\in E_{(n_{j+1},\ldots,n_k)}:(\Jet_{(n_1,\ldots,n_j)}(a),b)\in C^{\mathcal L}\}.
\]
Since \(a\in\Pi_j(C)\), the set \(D_a\) is a nonempty fiber of the o-minimal cell \(C^{\mathcal L}\), and hence is an o-minimal cell. For every \(b\in M^{k-j}\), we have \(b\in C_a\) if and only if \(\Jet_{(n_{j+1},\ldots,n_k)}(b)\in D_a\). Thus \(D_a\) is the natural candidate source cell for \(C_a\). Its layout is obtained by substituting the realized jet \(\Jet_{(n_1,\ldots,n_j)}(a)\) into the last \(k-j\) blocks of \(\mathcal F_C\), and it is prepared.

It remains to show that \(\Jet_{(n_{j+1},\ldots,n_k)}(C_a)\) is dense in \(D_a\). We prove this by induction on \(k-j\). If \(k-j=1\), then \(D_a\) has a single prepared block. Its band layers give open conditions on the corresponding finite jets, and density follows from genericity. Its graph layers are forced successively by the chain rule and Lemma~\ref{Going up graph uniqueness}. Hence the realized jets of \(C_a\) are dense in \(D_a\).

Assume the result for fibers in at most \(i\) variables, and suppose \(k-j=i+1\). Let \(U\subseteq D_a\) be a nonempty relatively open subset. After shrinking \(U\), we may assume that it is compatible with the last block of the layout. Let \(\widetilde D_a\) be the projection of \(D_a\) onto its first \(i\) remaining blocks. Then the projection of \(U\) to \(\widetilde D_a\) is a nonempty relatively open subset. By the inductive hypothesis, it contains a realized jet \(\Jet_{(n_{j+1},\ldots,n_{j+i})}(b)\), with \(b\in M^i\) and \((a,b)\in\Pi_{j+i}(C)\). The fiber of \(U\) over this realized jet is a nonempty relatively open subset of the one-block source cell of \(C_{(a,b)}\). By the one-variable case, this fiber contains a realized jet \(\Jet_{n_k}(c)\), with \(c\in C_{(a,b)}\). Therefore \(\Jet_{(n_{j+1},\ldots,n_k)}(b,c)\in U\). Since \(U\) was arbitrary, \(\Jet_{(n_{j+1},\ldots,n_k)}(C_a)\) is dense in \(D_a\).

Thus \(D_a\) is a source cell of \(C_a\). The required smoothness is inherited from the source layout \(\mathcal F_C\), and the induced layout remains prepared.
\end{proof}

The preceding results give the source configurations used in the sequel. For a sufficiently smooth \(\delta\)-cell, any chosen source layout determines the canonical source cell and its complete configuration. The lower configurations are obtained by coordinate projection, while the extended configurations are obtained by iterating the chain-rule construction. Moreover, projections and fibers over realized points inherit prepared source layouts, and their source cells are obtained by projecting the relevant blocks or by substituting the corresponding realized jets. This is the higher-dimensional analogue of the complete configuration constructed in the one-variable case.

\section{\texorpdfstring{\(\delta\)}{Delta}-Curves}

We now study curve selection for the \(\delta\)-topology. In this section we assume that \(T\) admits \(\mathcal C^\infty\)-cell decomposition. This assumption is used to avoid repeatedly refining source cells at each finite level. The arguments below are local on germs of curves and, at any fixed finite level, require only finite smoothness. Thus, in settings where only \(\mathcal C^p\)-cell decomposition is available, the same arguments may be applied after making sufficiently smooth local refinements for the finite level under consideration. Under the \(\mathcal C^\infty\)-cell decomposition assumption, every \(\delta\)-cell considered below is taken with a sufficiently smooth source layout, and all extended configurations are formed with respect to this chosen layout. Unless otherwise specified, closure means Euclidean closure.

The first point is that the \(\delta\)-closure of a \(\delta\)-cell can be tested on its finite configurations.

\begin{lemma}\label{Finite configurations characterize delta closure}
Let \(C\subseteq M^k\) be a \(\delta\)-cell of order \(n\in\mathbb N^k\), and let \(a\in M^k\). Recall that, for \(r\in\mathbb N\), the notation \(n+r\) means adding \(r\) to each coordinate of \(n\). Then \(a\in\cl_{n+r}(C)\) if and only if \(\Jet_{n+r}(a)\in\cl(\Jet_{n+r}(C))\) in \(E_{n+r}\). Consequently, \(a\in\cl_\delta(C)\) if and only if \(a\in\cl_{n+r}(C)\) for all \(r\in\mathbb N\). Equivalently, \(a\in\cl_\delta(C)\) if and only if \(\Jet_{n+r}(a)\in\cl(\Jet_{n+r}(C))\) for all \(r\in\mathbb N\).
\end{lemma}

\begin{proof}
By definition, the topology \(\mathcal T_{n+r}\) on \(M^k\) is induced by the embedding \(\Jet_{n+r}:M^k\to E_{n+r}\), where \(E_{n+r}\) is equipped with its Euclidean topology. Hence \(a\) lies in the \(\mathcal T_{n+r}\)-closure of \(C\) if and only if every Euclidean neighbourhood of \(\Jet_{n+r}(a)\) in \(E_{n+r}\) intersects \(\Jet_{n+r}(C)\). This is precisely the condition that \(\Jet_{n+r}(a)\in\cl(\Jet_{n+r}(C))\).

Since \(C\) has order \(n\), the finite refinements relevant to \(C\) are the topologies \(\mathcal T_{n+r}\), for \(r\in\mathbb N\), and \(\mathcal T_\delta\) is generated by these refinements. Therefore \(a\) belongs to the \(\delta\)-closure of \(C\) if and only if \(a\) belongs to the \(\mathcal T_{n+r}\)-closure of \(C\) for every \(r\in\mathbb N\). The final equivalence follows from the first paragraph.
\end{proof}

Thus, to understand the closure of a \(\delta\)-cell with respect to the \(\delta\)-topology, it is enough to understand the Euclidean closures of its successive finite configurations. In particular, if \(C\) has order \(n\), then
\[
\cl_\delta(C)=\bigcap_{r\in\mathbb N}\cl_{n+r}(C).
\]

It is tempting to conclude that the first closure in this descending chain already gives the \(\delta\)-closure, that is, whether \(\cl_n(C)=\cl_\delta(C)\) for a \(\delta\)-cell of order \(n\). The following example shows that this need not be the case.

\begin{example}\label{Example finite closure does not stabilize immediately}
Consider the \(\mathcal L^\delta\)-definable set \(C\subseteq M\) defined by
\[
(\delta x)^2-x=0\quad\text{and}\quad x>0.
\]
Then \(C\) has order \(1\), and \(0\in\cl_1(C)\). Indeed, in the first jet space, the source cell of \(C\) is the curve \(y^2=x\) with \(x>0\), whose Euclidean closure contains \((0,0)=\Jet_1(0)\). However, differentiating \((\delta x)^2=x\) gives \(2\delta x\delta^2x=\delta x\). Since \(x>0\) implies \(\delta x\neq0\) on \(C\), the set \(C\) is also defined by
\[
(\delta x)^2-x=0,\quad x>0,\quad\text{and}\quad \delta^2x=\frac12.
\]
Thus \(\Jet_2(C)\) is contained in the hyperplane on which the last coordinate is \(\frac12\), whereas \(\Jet_2(0)=(0,0,0)\). Hence \(0\notin\cl_2(C)\).
\end{example}

The example shows that higher configurations may impose new closure conditions which are invisible at the order of the original source formula. Thus a curve selection principle for the \(\delta\)-topology must control curves simultaneously through all finite configurations.

Before defining \(\delta\)-curves, we record why ordinary curve selection inside the underlying \(\mathcal L^\delta\)-definable set is not the right formulation.

\begin{proposition}\label{Non L definable one variable delta cells are dense and codense}
Let \(C\subseteq M\) be a sufficiently smooth \(\delta\)-cell which is not \(\mathcal L\)-definable. Then \(C\) is dense and codense in its Euclidean closure \(\cl(C)\).
\end{proposition}

\begin{proof}
Let \(n=\ord(C)\). Since \(C\) is not \(\mathcal L\)-definable, we have \(n>0\). By Proposition~\ref{One variable source cell shape}, either \(C^{\mathcal L}\) is open in \(M^{n+1}\), or \(C^{\mathcal L}=\Gamma(f)\) for some \(\mathcal L(M)\)-definable \(\mathcal C^1\)-function \(f:\Pi_{n-1}(C^{\mathcal L})\to M\), where \(\Pi_{n-1}(C^{\mathcal L})\) is an open cell in \(M^n\). In both cases, \(\Jet_n(C)\) is dense in \(C^{\mathcal L}\). Hence \(C=\Pi_0(\Jet_n(C))\) is dense in \(\Pi_0(C^{\mathcal L})\), and therefore dense in \(\cl(C)\).

We prove codensity. First suppose that \(C^{\mathcal L}\) is open in \(M^{n+1}\). Since \(C^{\mathcal L}\) is an open cell and \(n>0\), its final layer over \(\Pi_{n-1}(C^{\mathcal L})\) is a band \((f,g)_{\Pi_{n-1}(C^{\mathcal L})}\), where \(f<g\) are continuous \(\mathcal L(M)\)-definable boundary functions, allowing \(f=-\infty\) and \(g=+\infty\). By minimality of \(n=\ord(C)\), this final layer cannot be the full band \((-\infty,+\infty)_{\Pi_{n-1}(C^{\mathcal L})}\); otherwise the source formula would have order at most \(n-1\). Hence at least one of \(f,g\) is finite. Assume, for definiteness, that \(g\) is finite; the case where \(f\) is finite is analogous.

Choose a continuous \(\mathcal L(M)\)-definable function \(h:\Pi_{n-1}(C^{\mathcal L})\to M\) such that \(h>g\), and let \(D:=\Gamma(h)\). Then \(D\cap C^{\mathcal L}=\varnothing\), and \(D\) has the same base as the final layer of \(C^{\mathcal L}\). By Proposition~\ref{Complete configuration going up}, the graph \(D\) is a source cell for
\[
B:=\{x\in M:\Jet_{n-1}(x)\in\Pi_{n-1}(C^{\mathcal L})\text{ and }\delta^n x=h(\Jet_{n-1}(x))\}.
\]
Thus \(\Jet_n(B)\) is dense in \(D\). Since \(D\) and \(C^{\mathcal L}\) have the same base, \(B\) and \(C\) have the same Euclidean closure after projection to the first coordinate. Moreover \(B\cap C=\varnothing\). Hence \(C\) is codense in \(\cl(C)\).

Now suppose that \(C^{\mathcal L}=\Gamma(f)\). Choose an \(\mathcal L(M)\)-definable \(\mathcal C^1\)-function \(g:\Pi_{n-1}(C^{\mathcal L})\to M\) such that \(g>f\), and set
\[
B:=\{x\in M:\Jet_{n-1}(x)\in\Pi_{n-1}(C^{\mathcal L})\text{ and }\delta^n x=g(\Jet_{n-1}(x))\}.
\]
Then \(B\cap C=\varnothing\). By Proposition~\ref{Complete configuration going up}, \(\Gamma(g)\) is a source cell for \(B\), so \(\Jet_n(B)\) is dense in \(\Gamma(g)\). Since \(\Gamma(f)\) and \(\Gamma(g)\) have the same base, the sets \(B\) and \(C\) have the same Euclidean closure after projection to the first coordinate. Therefore \(B\) is dense in \(\cl(C)\), and \(C\) is codense in \(\cl(C)\).
\end{proof}

The preceding proposition shows that one cannot expect an ordinary curve selection statement inside an arbitrary \(\mathcal L^\delta\)-definable set. We therefore pass to sufficiently smooth \(\delta\)-cells and use their complete configurations. Let \(C\subseteq M^k\) be a \(\mathcal C^\infty\) \(\delta\)-cell of order \(n\in\mathbb N^k\), and let \((C_{n+r})_{r\in\mathbb N}\) be the extended configuration of \(C\). If \(a\in\cl_\delta(C)\), then by Lemma~\ref{Finite configurations characterize delta closure}, \(\Jet_{n+r}(a)\in\cl(C_{n+r})\) for every \(r\in\mathbb N\). Hence, by the o-minimal curve selection lemma, for each \(r\in\mathbb N\) there is an \(\mathcal L\)-definable continuous curve \(\gamma_r:(0,1)\to C_{n+r}\) such that \(\lim_{t\to0}\gamma_r(t)=\Jet_{n+r}(a)\).

Thus it is natural to select curves in the successive configurations of \(C\). The difficulty is that these configurations live in different ambient spaces, and the selected curves must be compatible under the natural projections between finite jet spaces.

\begin{definition}\label{Definition of delta curve}
Let \(C\subseteq M^k\) be a \(\mathcal C^\infty\) \(\delta\)-cell of order \(n\in\mathbb N^k\), and let \(a\in\cl_n(C)\). A \emph{\(\delta\)-curve} in \(C\) at \(a\) is a sequence \((\gamma_r)_{r\in\mathbb N}\) of continuous curves \(\gamma_r:(0,1)\to C_{n+r}\) such that \(\gamma_0\) converges to \(\Jet_n(a)\) and, for all \(r\leq s\), one has \(\Pi_{n+r}\circ\gamma_s=\gamma_r\). We say that \((\gamma_r)_{r\in\mathbb N}\) is \emph{definable} if each \(\gamma_r\) is \(\mathcal L\)-definable.

For \(\ell\in\mathbb N\), we say that \((\gamma_r)_{r\in\mathbb N}\) \emph{converges to \(a\) up to level \(\ell\)} if \(\lim_{t\to0}\gamma_r(t)=\Jet_{n+r}(a)\) for every \(0\leq r\leq \ell\). We say that \((\gamma_r)_{r\in\mathbb N}\) \emph{converges to \(a\)} if it converges to \(a\) up to level \(\ell\) for every \(\ell\in\mathbb N\).
\end{definition}

Fix a \(\mathcal C^\infty\) \(\delta\)-cell \(C\subseteq M^k\) of order \(n\in\mathbb N^k\), and let \((C_{n+r})_{r\in\mathbb N}\) be its extended configuration. By Lemma~\ref{Finite configurations characterize delta closure} and the o-minimal curve selection lemma, if \(a\in\cl_\delta(C)\), then for each \(r\in\mathbb N\) there exists an \(\mathcal L\)-definable continuous curve \(\sigma_r:(0,1)\to C_{n+r}\) such that \(\lim_{t\to0}\sigma_r(t)=\Jet_{n+r}(a)\). For each \(r\in\mathbb N\) and each \(0\leq s\leq r\), set \(\sigma_{r,s}:=\Pi_{n+s}\circ\sigma_r\). Then \(\sigma_{r,s}:(0,1)\to C_{n+s}\) and \(\lim_{t\to0}\sigma_{r,s}(t)=\Jet_{n+s}(a)\). Thus, for each finite level \(r\), there is a finite compatible system of curves witnessing convergence up to level \(r\). The remaining task is to replace these finite systems, as \(r\) varies, by a single definable \(\delta\)-curve whose canonical lifts converge to \(a\) at every finite level.

Given an \(\mathcal L\)-definable curve \(\gamma:(0,1)\to C_n\) such that \(\lim_{t\to0}\gamma(t)=\Jet_n(a)\), we now construct a compatible sequence of curves in the extended configuration of \(C\). Set \(\gamma_0:=\gamma\). Suppose that \(\gamma_r:(0,1)\to C_{n+r}\) has been constructed. Write
\[
\gamma_r=(\gamma_{r,1},\ldots,\gamma_{r,k}),
\]
where \(\gamma_{r,j}\) is the \(j^{\mathrm{th}}\) jet block of \(\gamma_r\), with coordinates up to order \(n_j+r\). We define \(\gamma_{r+1}\) block by block. Fix \(1\leq j\leq k\). If the next layer in the \(j^{\mathrm{th}}\) block of the extended configuration is a graph, say given by an \(\mathcal L(M)\)-definable function \(G_{j,r}\) on the corresponding lower configuration, then set
\[
\gamma_{r+1,j}:=(\gamma_{r,j},G_{j,r}\circ\gamma_r).
\]
If the next layer in the \(j^{\mathrm{th}}\) block is a band, then by construction of the extended configuration it is the full band, and we set
\[
\gamma_{r+1,j}:=(\gamma_{r,j},\delta^{n_j+r+1}a_j).
\]
Finally define \(\gamma_{r+1}:=(\gamma_{r+1,1},\ldots,\gamma_{r+1,k})\). Then \(\gamma_{r+1}:(0,1)\to C_{n+r+1}\) and \(\Pi_{n+r}\circ\gamma_{r+1}=\gamma_r\). Iterating this construction produces a compatible sequence \((\gamma_r)_{r\in\mathbb N}\).

\begin{definition}\label{Definition of canonical delta curve}
Using the notation above, the compatible sequence \((\gamma_r)_{r\in\mathbb N}\) constructed from \(\gamma\) is called the \emph{canonical \(\delta\)-curve} induced by \(\gamma\).
\end{definition}

The construction must be checked to land in the extended configuration at every level.

\begin{proposition}\label{Canonical delta curve is a delta curve}
Let \(C\subseteq M^k\) be a \(\mathcal C^\infty\) \(\delta\)-cell of order \(n\in\mathbb N^k\), let \(a\in\cl_n(C)\), and let \(\gamma:(0,1)\to C_n\) be a continuous curve such that \(\lim_{t\to0}\gamma(t)=\Jet_n(a)\). Then the canonical \(\delta\)-curve induced by \(\gamma\) is a \(\delta\)-curve in \(C\) at \(a\).
\end{proposition}

\begin{proof}
We prove by induction on \(r\) that \(\gamma_r:(0,1)\to C_{n+r}\) and that \(\Pi_{n+s}\circ\gamma_r=\gamma_s\) for all \(0\leq s\leq r\). The case \(r=0\) holds by assumption, since \(\gamma_0=\gamma\) takes values in \(C_n\).

Suppose the claim holds for some \(r\). The curve \(\gamma_{r+1}\) is defined block by block from \(\gamma_r\). If the next layer in a given block is a graph, then the new coordinate is obtained by applying the defining function of that graph to \(\gamma_r\), so the resulting block lies in the graph layer of the extended configuration. If the next layer is a band, then by the construction of the extended configuration it is the full band, and the new coordinate is the constant function \(\delta^{n_j+r+1}a_j\), hence it lies in that band layer. Therefore \(\gamma_{r+1}\) takes values in \(C_{n+r+1}\).

Moreover, \(\gamma_{r+1}\) is obtained from \(\gamma_r\) by appending one new coordinate to each jet block. Hence \(\Pi_{n+r}\circ\gamma_{r+1}=\gamma_r\). Composing with the projections \(\Pi_{n+s}\) for \(s\leq r\) gives \(\Pi_{n+s}\circ\gamma_{r+1}=\gamma_s\). This completes the induction.
\end{proof}

The canonical lift is useful because convergence of an arbitrary compatible lift is equivalent to convergence of the canonical lift induced by its initial curve.

\begin{proposition}\label{Canonical delta curve detects convergence}
Let \(C\subseteq M^k\) be a \(\mathcal C^\infty\) \(\delta\)-cell of order \(n\in\mathbb N^k\), let \(a\in\cl_n(C)\), and let \((\sigma_r)_{r\in\mathbb N}\) be a \(\delta\)-curve in \(C\) at \(a\). Let \((\gamma_r)_{r\in\mathbb N}\) be the canonical \(\delta\)-curve induced by the initial curve \(\sigma_0\). Then \((\sigma_r)_{r\in\mathbb N}\) converges to \(a\) if and only if \((\gamma_r)_{r\in\mathbb N}\) converges to \(a\).
\end{proposition}

\begin{proof}
For each \(r\in\mathbb N\), let \(P(r)\) be the statement
\[
\lim_{t\to0}\sigma_r(t)=\Jet_{n+r}(a)
\quad\Longleftrightarrow\quad
\lim_{t\to0}\gamma_r(t)=\Jet_{n+r}(a).
\]
We prove \(P(r)\) by induction on \(r\). The case \(r=0\) is immediate, since \(\gamma_0=\sigma_0\).

Assume \(P(r)\). We prove \(P(r+1)\). Write
\[
\Jet_{n+r+1}(a)=\bigl(\Jet_{n+r}(a),b_{r+1}(a)\bigr),
\]
where \(b_{r+1}(a)\) denotes the tuple of the new jet coordinates added at level \(r+1\). Likewise write
\[
\sigma_{r+1}=(\sigma_r,\sigma'_{r+1})
\quad\text{and}\quad
\gamma_{r+1}=(\gamma_r,\gamma'_{r+1}).
\]
For each new coordinate, there are two cases.

If the corresponding layer of \(C_{n+r+1}\) over \(C_{n+r}\) is a graph, then there is an \(\mathcal L(M)\)-definable continuous function \(G\) on the relevant lower configuration such that, for both curves,
\[
\sigma'_{r+1,j}=G\circ\sigma_r,
\qquad
\gamma'_{r+1,j}=G\circ\gamma_r.
\]
Moreover, since \(\Jet_{n+r+1}(a)\in\cl(C_{n+r+1})\), the corresponding new coordinate of \(\Jet_{n+r+1}(a)\) is \(b_{r+1,j}(a)=G(\Jet_{n+r}(a))\). Hence convergence at level \(r\) implies convergence of this graph coordinate at level \(r+1\), for both \(\sigma\) and \(\gamma\), by continuity of \(G\).

If the corresponding layer is a full band, then the canonical lift is defined by \(\gamma'_{r+1,j}(t)=b_{r+1,j}(a)\) for all \(t\). Thus this coordinate of \(\gamma_{r+1}\) always converges to the correct coordinate of \(\Jet_{n+r+1}(a)\). For \(\sigma_{r+1}\), convergence to \(\Jet_{n+r+1}(a)\) is precisely the assertion that every such full-band coordinate also converges to its corresponding coordinate \(b_{r+1,j}(a)\).

Combining the graph and full-band coordinates, we see that convergence of \(\sigma_{r+1}\) to \(\Jet_{n+r+1}(a)\) is equivalent to convergence of \(\gamma_{r+1}\) to \(\Jet_{n+r+1}(a)\), using the induction hypothesis at level \(r\). Thus \(P(r+1)\) holds. Therefore \(P(r)\) holds for all \(r\), and the proposition follows.
\end{proof}

Thus it suffices to work with ordinary \(\mathcal L\)-definable curves in the source cell \(C^{\mathcal L}\). The goal is to find such a curve \(\gamma:(0,1)\to C^{\mathcal L}\) whose induced canonical \(\delta\)-curve converges to \(a\). Since this convergence condition depends only on the behavior of \(\gamma\) near \(a\), we pass to germs of curves.

\section{Germ Spaces}

Fix \(k\in\mathbb N\) and \(a\in M^k\). A \emph{curve at \(a\)} is a continuous function \(\gamma:(0,1)\to M^k\) such that \(\lim_{t\to0}\gamma(t)=a\). If \(\gamma\) and \(\sigma\) are two curves at \(a\), we write \(\gamma\sim\sigma\) if there exists \(\varepsilon>0\) such that \(B(a,\varepsilon)\cap\gamma((0,1))=B(a,\varepsilon)\cap\sigma((0,1))\). Equivalently, the images of \(\gamma\) and \(\sigma\) agree in some Euclidean neighbourhood of \(a\). We denote by \([\gamma]\) the equivalence class of \(\gamma\).

Let \(\mathcal K:=\mathcal H_{0^+}\) be the Hardy field of germs at \(0^+\) of \(\mathcal L(M)\)-definable unary functions. Let \(\mathcal G_a\) be the set of germs of \(\mathcal L\)-definable curves in \(M^k\) at \(a\), and omit the subscript \(a\) when it is clear from context. Let \(d:M^k\times M^k\to M\) be the sup metric, \(d(x,y)=\max_{i=1}^k |x_i-y_i|\). We define a \(\mathcal K\)-valued function \(d_{\mathcal G}:\mathcal G\times\mathcal G\to\mathcal K\) by
\[
d_{\mathcal G}([\gamma],[\sigma])=[f]_{\mathcal K},
\]
where, for all sufficiently small \(t>0\),
\[
f(t)=D\bigl(B(a,t)\cap\gamma((0,1)),B(a,t)\cap\sigma((0,1))\bigr),
\]
and \(D\) is the Hausdorff distance induced by \(d\).

\begin{remark}\label{Frontier representation of germ distance}
By the monotonicity theorem, after replacing \(\gamma\) and \(\sigma\) by equivalent representatives and shrinking the domain if necessary, we may assume that the functions \(d(\gamma(t),a)\) and \(d(\sigma(t),a)\) are strictly increasing as functions of \(t\) and tend to \(0\) as \(t\to0^+\). Thus, for all sufficiently small \(\varepsilon>0\), each image \(\gamma((0,1))\) and \(\sigma((0,1))\) intersects the frontier \(\partial B(a,\varepsilon)\) in exactly one point. Under this choice of representatives, the Hausdorff distance between the two curve germs inside \(B(a,\varepsilon)\) is represented by the distance between these frontier points.
\end{remark}

The preceding construction turns the set of definable curve germs at \(a\) into a metric object over the Hardy field \(\mathcal K\).

\begin{proposition}\label{Germ distance is a K metric}
The function \(d_{\mathcal G}\) is a well-defined \(\mathcal K\)-metric on \(\mathcal G\).
\end{proposition}

\begin{proof}
We first show that \(d_{\mathcal G}\) is well-defined. Suppose that \(\gamma_1\sim\gamma_2\) and \(\sigma_1\sim\sigma_2\). Then, for all sufficiently small \(t>0\), the intersections of \(\gamma_1((0,1))\) and \(\gamma_2((0,1))\) with \(B(a,t)\) agree, and similarly for \(\sigma_1\) and \(\sigma_2\). Hence the corresponding Hausdorff-distance functions agree for all sufficiently small \(t>0\), and therefore define the same element of \(\mathcal K\).

Suppose next that \(d_{\mathcal G}([\gamma],[\sigma])=0\). Then, for all sufficiently small \(t>0\), the Hausdorff distance between \(B(a,t)\cap\gamma((0,1))\) and \(B(a,t)\cap\sigma((0,1))\) is zero. Hence these two definable curve germs have the same local closure near \(a\). By o-minimality, two definable curve germs at \(a\) with the same local closure have the same local image after shrinking. Therefore \(\gamma\sim\sigma\), and so \([\gamma]=[\sigma]\). The converse is immediate from the definition of \(\sim\).

Symmetry follows from symmetry of the Hausdorff distance. Finally, let \([\gamma_1]\), \([\gamma_2]\), and \([\gamma_3]\) be three curve germs. For sufficiently small \(t>0\), set \(A_i(t):=B(a,t)\cap\gamma_i((0,1))\) for \(i=1,2,3\). The triangle inequality for the Hausdorff distance gives
\[
D(A_1(t),A_3(t))\leq D(A_1(t),A_2(t))+D(A_2(t),A_3(t)).
\]
Passing to germs in \(\mathcal K\) gives the triangle inequality for \(d_{\mathcal G}\). Thus \(d_{\mathcal G}\) is a well-defined \(\mathcal K\)-metric on \(\mathcal G\).
\end{proof}

We next describe the local pieces of the germ space. The idea is to normalize a curve germ by choosing one coordinate which approaches \(a\) monotonically from one side, and then use that coordinate as the parameter.

For each \(i=1,\ldots,k\), define
\[
H_{a,i}^{+}:=\{b\in M^k:b_i>a_i\},
\qquad
H_{a,i}^{-}:=\{b\in M^k:b_i<a_i\}.
\]
Let \(\mathcal G_{a,i}^{+}\subseteq\mathcal G_a\) be the set of all equivalence classes of \(\mathcal L\)-definable curve germs at \(a\) whose representatives are eventually contained in \(H_{a,i}^{+}\), and define \(\mathcal G_{a,i}^{-}\) similarly. These sets cover \(\mathcal G_a\), since every nonconstant definable curve germ at \(a\) is eventually contained in one of the half spaces.

Let \(a=(a_1,\ldots,a_k)\in M^k\), and write \(a_{\widehat i}:=(a_1,\ldots,a_{i-1},a_{i+1},\ldots,a_k)\). Let \(\operatorname{ins}_i:M^{k-1}\times M\to M^k\) be the insertion map
\[
\operatorname{ins}_i(u,y):=(u_1,\ldots,u_{i-1},y,u_i,\ldots,u_{k-1}).
\]
For \([\gamma]\in\mathcal G_{a,i}^{+}\), by o-minimal monotonicity, after replacing \(\gamma\) by an equivalent representative and shrinking its domain, we may assume that the \(i^{\mathrm{th}}\) coordinate function \(\gamma_i\) is continuous, injective, and satisfies \(\gamma_i(t)>a_i\) for all \(t\). Set \(s=\gamma_i(t)-a_i\). Then \(\gamma\) has a unique normalized representative of the form
\[
\widehat\gamma(s)=\operatorname{ins}_i(u_\gamma(s),a_i+s),
\]
where \(u_\gamma:(0,\varepsilon)\to M^{k-1}\) is \(\mathcal L(M)\)-definable and converges to \(a_{\widehat i}\). Define
\[
\iota_{a,i}^{+}:\mathcal G_{a,i}^{+}\to\mathfrak m_{\mathcal K}^{k-1}
\]
by
\[
\iota_{a,i}^{+}([\gamma]):=[u_\gamma-a_{\widehat i}]_{\mathcal K}.
\]
Similarly, for \([\gamma]\in\mathcal G_{a,i}^{-}\), we use the normalization \(\widehat\gamma(s)=\operatorname{ins}_i(u_\gamma(s),a_i-s)\), and define \(\iota_{a,i}^{-}([\gamma]):=[u_\gamma-a_{\widehat i}]_{\mathcal K}\).

The following elementary lemma gives the uniform continuity estimate needed to compare the germ metric with the coordinate metric on \(\mathfrak m_{\mathcal K}^{k-1}\).

\begin{lemma}\label{Definable germs have moduli of continuity}
Let \(u:(0,\varepsilon)\to M^m\) be an \(\mathcal L(M)\)-definable continuous function such that \(\lim_{s\to0^+}u(s)\) exists in \(M^m\). Then there is an \(\mathcal L(M)\)-definable function \(\omega:(0,\varepsilon')\to M^{>0}\) such that \(\lim_{\eta\to0^+}\omega(\eta)=0\), and, for all sufficiently small \(s,s'>0\), if \(|s-s'|<\eta\), then \(d(u(s),u(s'))<\omega(\eta)\).
\end{lemma}

\begin{proof}
For sufficiently small \(\eta>0\), define
\[
\omega(\eta):=\sup\{d(u(s),u(s')):0<s,s'<\varepsilon,\ |s-s'|<\eta\}.
\]
After shrinking \(\varepsilon\), this supremum exists in \(M\), because \(u\) has a limit at \(0\) and is bounded near \(0\). By definable completeness of o-minimal expansions of real closed fields, \(\omega\) is \(\mathcal L(M)\)-definable. Since \(u\) extends continuously to \(0\) by setting \(u(0)=\lim_{s\to0^+}u(s)\), continuity at \(0\) gives \(\omega(\eta)\to0\) as \(\eta\to0^+\).
\end{proof}

\begin{proposition}\label{Charts are Euclidean spaces over the Hardy field}
The map \(\iota_{a,i}^{+}\) is a homeomorphism from \(\mathcal G_{a,i}^{+}\), equipped with the subspace topology induced by the \(\mathcal K\)-metric \(d_{\mathcal G}\) on \(\mathcal G_a\), onto \(\mathfrak m_{\mathcal K}^{k-1}\), equipped with the sup metric \(d_K(x,y)=\max_{1\leq j\leq k-1}|x_j-y_j|\). The analogous statement holds for \(\iota_{a,i}^{-}\).
\end{proposition}

\begin{proof}
We prove the statement for \(\iota_{a,i}^{+}\); the proof for \(\iota_{a,i}^{-}\) is analogous. By the normalization above, every germ in \(\mathcal G_{a,i}^{+}\) has a unique normalized representative of the form
\[
\widehat\gamma(s)=\operatorname{ins}_i(u_\gamma(s),a_i+s),
\]
where \(u_\gamma:(0,\varepsilon)\to M^{k-1}\) is \(\mathcal L(M)\)-definable and converges to \(a_{\widehat i}\). The map \(\iota_{a,i}^{+}\) sends \([\gamma]\) to \([u_\gamma-a_{\widehat i}]_{\mathcal K}\in\mathfrak m_{\mathcal K}^{k-1}\). Conversely, every element \(w\in\mathfrak m_{\mathcal K}^{k-1}\) is represented by a definable function \(w(s)\) with \(w(s)\to0\), and the curve \(s\mapsto\operatorname{ins}_i(a_{\widehat i}+w(s),a_i+s)\) defines an element of \(\mathcal G_{a,i}^{+}\) mapping to \(w\). Hence \(\iota_{a,i}^{+}\) is bijective.

It remains to compare the topology induced by \(d_{\mathcal G}\) with the sup metric topology on \(\mathfrak m_{\mathcal K}^{k-1}\). Since both metrics are sup metrics, it is enough to compare one omitted coordinate at a time. Thus we reduce to the case \(k=2\) and \(i=1\). After translating \(a\) to \((0,0)\), normalized representatives have the form
\[
\widehat\gamma(s)=(s,u(s)),\qquad \widehat\sigma(s)=(s,v(s)),
\]
where \(u,v:(0,\varepsilon)\to M\) are \(\mathcal L(M)\)-definable and tend to \(0\) as \(s\to0^+\).

We first prove continuity of \((\iota_{a,1}^{+})^{-1}\). Let \([\rho]\in\mathcal K^{>0}\) be a positive infinitesimal germ and suppose that \(|u-v|<\rho\) as germs. By Lemma~\ref{Definable germs have moduli of continuity}, the graphs of \(u\) and \(v\), truncated by sufficiently small sup-balls around \(0\), remain within a positive infinitesimal Hausdorff distance of each other. Hence \(d_{\mathcal G}([\gamma],[\sigma])\) is infinitesimal whenever \(d_K(\iota_{a,1}^{+}([\gamma]),\iota_{a,1}^{+}([\sigma]))\) is infinitesimal. This proves continuity of \((\iota_{a,1}^{+})^{-1}\).

Conversely, we prove continuity of \(\iota_{a,1}^{+}\). Let \([\rho]\in\mathcal K^{>0}\) be a positive infinitesimal germ. We must find a positive infinitesimal germ \([\eta]\) such that \(d_{\mathcal G}([\gamma],[\sigma])<[\eta]\) implies \(|u-v|<\rho\) as germs. By o-minimality, after shrinking the domain, \(u\) is either constant or strictly monotone; in the latter case, after applying \(\mathcal C^2\)-cell decomposition, it is affine, strictly convex, or strictly concave. If \(u\) is constant, the claim is immediate.

Suppose first that \(u\) is strictly convex, or affine with derivative eventually of absolute value at most \(1\). Define
\[
h(s):=\max\left\{s+\frac{u(s)}2,\ u(s)+\frac{\rho(s)}2\right\}.
\]
After shrinking the domain, \(h(s)>u(s)\) and \(h(s)-u(s)\) is a positive infinitesimal germ. Set \([\eta]:=[h-u]_{\mathcal K}\). If \(d_{\mathcal G}([\gamma],[\sigma])<[\eta]\), then the graph of \(v\) cannot leave the vertical strip \(u(s)-\rho(s)<v(s)<u(s)+\rho(s)\) as a germ; otherwise convexity, or the affine slope bound, would produce a separation of at least \(h(s)-u(s)\) inside a sufficiently small sup-ball. Hence \(|u-v|<\rho\) as germs.

Now suppose that \(u\) is strictly concave, or affine with derivative eventually of absolute value greater than \(1\). Since \(u\) is strictly monotone, it has a definable inverse near \(0\). Define \(h\) by
\[
h^{-1}(s):=u^{-1}(s)+\frac{\rho(s)}2
\]
on a sufficiently small interval. Again \(h(s)>u(s)\) and \(h(s)-u(s)\) is a positive infinitesimal germ. With \([\eta]:=[h-u]_{\mathcal K}\), the same argument in the inverse parameter shows that \(d_{\mathcal G}([\gamma],[\sigma])<[\eta]\) implies \(|u-v|<\rho\) as germs.

Thus \(\iota_{a,1}^{+}\) and its inverse are continuous. Therefore \(\iota_{a,i}^{+}\) is a homeomorphism. The proof for \(\iota_{a,i}^{-}\) is the same, using normalized representatives of the form \(\widehat\gamma(s)=\operatorname{ins}_i(u_\gamma(s),a_i-s)\).
\end{proof}

The preceding proposition shows that the pieces \(\mathcal G_{a,i}^{+}\) and \(\mathcal G_{a,i}^{-}\) give local coordinates on the germ space. If two such pieces overlap, the corresponding coordinate changes are induced by definable reparameterizations of curve germs.

Let \(\mathcal G_{a,i}^{\star}\) and \(\mathcal G_{a,j}^{\diamond}\), with \(\star,\diamond\in\{+,-\}\), have nonempty intersection. On this overlap, define
\[
\tau_{i,j}^{\star,\diamond}
:=
\iota_{a,j}^{\diamond}\circ(\iota_{a,i}^{\star})^{-1}.
\]
The map \(\tau_{i,j}^{\star,\diamond}\) is induced by reparameterizing the same definable curve germ using two different monotone coordinates.

By Proposition~\ref{Charts are Euclidean spaces over the Hardy field}, each map \(\iota_{a,i}^{+}\) and \(\iota_{a,i}^{-}\) identifies its domain with \(\mathfrak m_{\mathcal K}^{k-1}\). Therefore the family
\[
\mathcal H_a:=
\left\{
(\mathcal G_{a,i}^{+},\iota_{a,i}^{+}),
(\mathcal G_{a,i}^{-},\iota_{a,i}^{-})
:1\leq i\leq k
\right\}
\]
is an atlas on \(\mathcal G_a\). Its elements are the charts of \(\mathcal G_a\), and the maps \(\tau_{i,j}^{\star,\diamond}\) are the transition maps. Thus \(\mathcal G_a\) is a \(K\)-manifold modeled on \(\mathfrak m_{\mathcal K}^{k-1}\).

We shall also use the following global direction map. Define
\[
\iota:\mathcal G_a\to\mathbb P^{k-1}(K)
\]
as follows. On \(\mathcal G_{a,i}^{+}\), let \(\iota([\gamma])\) be the projective class of the germ
\[
\widehat\gamma^{i,+}(s)-a\in K^k,
\]
and define it similarly on \(\mathcal G_{a,i}^{-}\), using the normalized representative from the negative side. This is well-defined on overlaps, since different normalizations are obtained from one another by definable reparameterizations and hence determine the same projective germ direction.

Let \(A\subseteq M^k\) be an \(\mathcal L\)-definable set such that \(a\in\cl(A)\). Suppose \(A_{a,i}^{+}:=A\cap H_{a,i}^{+}\) is nonempty, and let
\[
\mathcal A_{a,i}^{+}:=
\{[\gamma]\in\mathcal G_{a,i}^{+}:\gamma\text{ is eventually contained in }A\}.
\]
Let \(\epsilon:=[s\mapsto s]_{\mathcal K}\in\mathfrak m_{\mathcal K}^{>0}\). Identifying \(M\) with the constant germs in \(K\), the coordinate description above gives
\[
\iota_{a,i}^{+}(\mathcal A_{a,i}^{+})
=
\left\{
w\in\mathfrak m_{\mathcal K}^{k-1}:
\operatorname{ins}_i(a_{\widehat i}+w,\ a_i+\epsilon)\in A^{\mathcal K}
\right\}.
\]
Indeed, if \(\widehat\gamma(s)=\operatorname{ins}_i(u_\gamma(s),a_i+s)\) is the normalized representative of \([\gamma]\), then \(\gamma\) is eventually contained in \(A\) if and only if \(\widehat\gamma\) is eventually contained in \(A\). By the Hardy-field interpretation of definable germs, this is equivalent to \(\operatorname{ins}_i([u_\gamma]_{\mathcal K},a_i+\epsilon)\in A^{\mathcal K}\). Since \(\iota_{a,i}^{+}([\gamma])=[u_\gamma-a_{\widehat i}]_{\mathcal K}\), this is precisely the displayed description.

The negative case is analogous. If
\[
\mathcal A_{a,i}^{-}:=
\{[\gamma]\in\mathcal G_{a,i}^{-}:\gamma\text{ is eventually contained in }A\},
\]
then
\[
\iota_{a,i}^{-}(\mathcal A_{a,i}^{-})
=
\left\{
w\in\mathfrak m_{\mathcal K}^{k-1}:
\operatorname{ins}_i(a_{\widehat i}+w,\ a_i-\epsilon)\in A^{\mathcal K}
\right\}.
\]
Thus each \(\mathcal A_{a,i}^{\pm}\) is definable in coordinates in the induced structure on \(\mathfrak m_{\mathcal K}^{k-1}\).

Let \(\mathcal A_a\) be the \(K\)-submanifold of \(\mathcal G_a\) whose coordinate pieces are the \(\mathcal A_{a,i}^{\pm}\), equipped with the restricted maps \(\iota_{a,i}^{\pm}|_{\mathcal A_{a,i}^{\pm}}\). In these coordinates, the pieces of \(\mathcal A_a\) are the definable subsets of \(\mathfrak m_{\mathcal K}^{k-1}\) described above.

\begin{remark}
Although each coordinate piece \(\mathcal A_{a,i}^{\pm}\) is definable after applying the corresponding map \(\iota_{a,i}^{\pm}\), it is not automatic that the whole germ space \(\mathcal A_a\) is definable as a \(K\)-manifold. This would require the transition maps between the pieces to be definable in the relevant induced structure. For the arguments below, we only use the definability of the individual coordinate pieces.
\end{remark}

\section{Abstract Curve Selection Lemma}

We now use the germ space constructed above to formulate an abstract curve selection lemma for the \(\delta\)-topology. Let \(C\subseteq M^k\) be a \(\mathcal C^\infty\) \(\delta\)-cell of order \(n\in\mathbb N^k\), and let \(a\in\cl_\delta(C)\). The ordinary curves from which we build canonical \(\delta\)-curves live in the source cell \(C^{\mathcal L}\subseteq E_n\), and therefore their germs are taken at the point \(\Jet_n(a)\in E_n\). Set
\[
N:=\sum_{j=1}^k(n_j+1).
\]

Let \(V\) be the convex hull of \(M\) in \(K\), and let \(\mathfrak m_{\mathcal K}\) be its maximal ideal. Since \(\mathcal K\succeq\mathcal M\) as an \(\mathcal L\)-structure, the pair \((\mathcal K,V)\) is a model of \(T_{\mathrm{convex}}\). By the quantifier elimination and weak o-minimality results for \(T\)-convex structures recalled in the preliminaries, we regard \((\mathcal K,V)\) as an \(\mathcal L_{\mathrm{convex}}\)-structure.

For each \(r\in\mathbb N\), Lemma~\ref{Finite configurations characterize delta closure} and the o-minimal curve selection lemma give an \(\mathcal L\)-definable curve germ in \(C^{\mathcal L}\) at \(\Jet_n(a)\) whose induced canonical \(\delta\)-curve converges to \(a\) up to level \(r\). After shrinking the representative, such a germ is eventually contained in one of the finitely many half-space pieces
\[
\mathcal G_{\Jet_n(a),i}^{\star},
\qquad
1\leq i\leq N,\quad \star\in\{+,-\}.
\]
Here the index \(i\) ranges over the coordinates of the source space \(E_n\). Moreover, if a source germ lies in \(\mathcal G_{\Jet_n(a),i}^{\star}\), then every level of its canonical lift lies in the half-space with the same coordinate direction and sign, now centered at the corresponding jet \(\Jet_{n+r}(a)\). Indeed, the old coordinates are preserved under the projection \(E_{n+r}\to E_n\), so the \(i^{\mathrm{th}}\) source coordinate keeps the same strict inequality throughout the lifted configurations.

Since there are only finitely many pairs \((i,\star)\), we may choose one pair such that the corresponding finite-level witnesses occur for arbitrarily large \(r\). Fix such a pair. Then for every \(r\in\mathbb N\) there is a witness in this same piece, because convergence up to a larger level implies convergence up to every smaller level.

We work in this fixed piece and use the map
\[
\iota_{\Jet_n(a),i}^{\star}:\mathcal G_{\Jet_n(a),i}^{\star}\to\mathfrak m_{\mathcal K}^{N-1}.
\]
For each \(r\in\mathbb N\), let \(\mathcal B_r(C,a)\) be the set of all germs \([\gamma]\in\mathcal G_{\Jet_n(a),i}^{\star}\) such that \(\gamma\) is eventually contained in \(C^{\mathcal L}\) and the canonical \(\delta\)-curve induced by \(\gamma\) converges to \(a\) up to level \(r\). By the choice of \((i,\star)\), each \(\mathcal B_r(C,a)\) is nonempty. Moreover,
\[
\mathcal B_{r+1}(C,a)\subseteq\mathcal B_r(C,a)
\]
for all \(r\in\mathbb N\).

\begin{lemma}\label{Existence of abstract curve germ}
With notation as above, there are a \(|K|^+\)-saturated elementary extension \((\mathcal K^*,V^*)\succeq(\mathcal K,V)\) in the language \(\mathcal L_{\mathrm{convex}}\), and an element \(\widetilde f\in(\mathfrak m_{\mathcal K^*})^{N-1}\), such that for every \(r\in\mathbb N\) and every \(\mathcal L(K)\)-definable open set \(U\subseteq K^{N-1}\), if \(\widetilde f\in U^{\mathcal K^*}\), then
\[
U\cap \iota_{\Jet_n(a),i}^{\star}(\mathcal B_r(C,a))\neq\emptyset.
\]
\end{lemma}

\begin{proof}
For each \(r\in\mathbb N\), set
\[
X_r:=\iota_{\Jet_n(a),i}^{\star}(\mathcal B_r(C,a)).
\]
Then \(X_r\subseteq\mathfrak m_{\mathcal K}^{N-1}\), the sets \(X_r\) are nonempty, and \(X_{r+1}\subseteq X_r\). Moreover, each \(X_r\) is \(\mathcal L_{\mathrm{convex}}(K)\)-definable in \((\mathcal K,V)\). Indeed, eventual containment in \(C^{\mathcal L}\) is definable in the germ coordinates by the coordinate description of definable subsets in the previous section, and convergence up to a fixed finite level is a finite condition on the canonical lift.

Let \((\mathcal K^*,V^*)\succeq(\mathcal K,V)\) be \(|K|^+\)-saturated. Since the decreasing family \((X_r)_{r\in\mathbb N}\) has the finite intersection property and has cardinality at most \(|K|\), saturation gives an element
\[
\widetilde f\in\bigcap_{r\in\mathbb N}X_r^{\mathcal K^*}.
\]
As each \(X_r\) is contained in \(\mathfrak m_{\mathcal K}^{N-1}\), we have \(\widetilde f\in(\mathfrak m_{\mathcal K^*})^{N-1}\).

Now let \(U\subseteq K^{N-1}\) be \(\mathcal L(K)\)-definable and open, with \(\widetilde f\in U^{\mathcal K^*}\). Since \(\widetilde f\in X_r^{\mathcal K^*}\), the set \(U^{\mathcal K^*}\cap X_r^{\mathcal K^*}\) is nonempty. By elementarity, \(U\cap X_r\neq\emptyset\). This is exactly the displayed conclusion.
\end{proof}

\begin{remark}\label{Remark on abstract curve germs}
The element \(\widetilde f\) should be understood as an abstract germ of an ordinary curve in the source cell \(C^{\mathcal L}\). Every \(\mathcal L(K)\)-definable neighbourhood of \(\widetilde f\) contains the coordinate germ of a genuine \(\mathcal L\)-definable curve whose induced canonical \(\delta\)-curve converges to \(a\) up to any prescribed finite level.
\end{remark}

\begin{definition}\label{Definition of abstract delta curve germ}
Let \(C\subseteq M^k\) be a \(\mathcal C^\infty\) \(\delta\)-cell of order \(n\in\mathbb N^k\), and let \(a\in M^k\). Set \(N:=\sum_{j=1}^k(n_j+1)\). An \emph{abstract curve germ in \(C\) at \(a\)} consists of a half-space piece \(\mathcal G_{\Jet_n(a),i}^{\star}\), a \(|K|^+\)-saturated elementary extension \((\mathcal K^*,V^*)\succeq(\mathcal K,V)\) as \(\mathcal L_{\mathrm{convex}}\)-structures, and an element \(\widetilde f\in(\mathfrak m_{\mathcal K^*})^{N-1}\), such that the corresponding abstract normalized point lies in the interpretation of the source cell. More explicitly, if \(\star=+\), then
\[
\operatorname{ins}_i\bigl(\Jet_n(a)_{\widehat i}+\widetilde f,\ \Jet_n(a)_i+\epsilon\bigr)
\in (C^{\mathcal L})^{\mathcal K^*},
\]
and if \(\star=-\), then
\[
\operatorname{ins}_i\bigl(\Jet_n(a)_{\widehat i}+\widetilde f,\ \Jet_n(a)_i-\epsilon\bigr)
\in (C^{\mathcal L})^{\mathcal K^*}.
\]
Here \(\epsilon=[s\mapsto s]_{\mathcal K^*}\) is the positive infinitesimal represented by the identity germ.

We say that this abstract curve germ \emph{converges to \(a\) along \(C\)} if, for every \(r\in\mathbb N\) and every \(\mathcal L(K)\)-definable open set \(U\subseteq K^{N-1}\) with \(\widetilde f\in U^{\mathcal K^*}\), there is an \(\mathcal L\)-definable curve germ \([\gamma]\in\mathcal G_{\Jet_n(a),i}^{\star}\) such that \(\gamma\) is eventually contained in \(C^{\mathcal L}\), \(\iota_{\Jet_n(a),i}^{\star}([\gamma])\in U\), and the canonical \(\delta\)-curve induced by \(\gamma\) converges to \(a\) up to level \(r\).
\end{definition}

\begin{theorem}\label{Abstract curve selection lemma}
Let \(C\subseteq M^k\) be a \(\mathcal C^\infty\) \(\delta\)-cell, and let \(a\in M^k\). Then \(a\in\cl_\delta(C)\) if and only if there exists an abstract curve germ in \(C\) at \(a\) which converges to \(a\) along \(C\).
\end{theorem}

\begin{proof}
Suppose first that \(a\in\cl_\delta(C)\). Lemma~\ref{Existence of abstract curve germ} gives an abstract element \(\widetilde f\) with the required approximation property. Since the sets \(\mathcal B_r(C,a)\) consist of germs eventually contained in \(C^{\mathcal L}\), the corresponding abstract normalized point lies in \((C^{\mathcal L})^{\mathcal K^*}\) by elementarity. Hence \(\widetilde f\) is an abstract curve germ in \(C\) at \(a\), and by construction it converges to \(a\) along \(C\).

Conversely, suppose that there exists an abstract curve germ in \(C\) at \(a\) which converges to \(a\) along \(C\). Fix \(r\in\mathbb N\) and choose an \(\mathcal L(K)\)-definable open neighbourhood \(U\) of its coordinate \(\widetilde f\). By convergence along \(C\), there is an \(\mathcal L\)-definable curve germ whose induced canonical \(\delta\)-curve converges to \(a\) up to level \(r\). Hence \(\Jet_{n+r}(a)\in\cl(C_{n+r})\). Since this holds for every \(r\in\mathbb N\), Lemma~\ref{Finite configurations characterize delta closure} gives \(a\in\cl_\delta(C)\).
\end{proof}

\begin{remark}\label{Meaning of abstract curve selection}
An abstract curve germ need not be represented by a curve in \(M\). Its membership in \(C\) is expressed by requiring its abstract normalized point to lie in the interpretation of the source cell \(C^{\mathcal L}\) in the saturated extension. Its convergence to \(a\) is expressed by approximation: every definable neighbourhood of the abstract germ contains genuine \(\mathcal L\)-definable curve germs in \(C^{\mathcal L}\) whose induced canonical \(\delta\)-curves converge to \(a\) to any prescribed finite level. Thus Theorem~\ref{Abstract curve selection lemma} is the abstract analogue of ordinary curve selection for the \(\delta\)-topology.
\end{remark}

\section{Concrete Representatives in \texorpdfstring{\(\CODF\)}{CODF}}

We now specialize to the case \((\mathbb R,\delta)\models\CODF\). The abstract curve selection lemma produces an abstract curve germ in a saturated elementary extension of the Hardy field. We first fix the notation for the fields of series used in this section. For \(m\in\mathbb N_{>0}\), \(\mathbb R((t^{1/m}))\) denotes the field of Laurent series in \(t^{1/m}\) over \(\mathbb R\), and
\[
\mathbb R((t^{1/\infty})):=\bigcup_{m\in\mathbb N_{>0}}\mathbb R((t^{1/m}))
\]
is the Puiseux series field over \(\mathbb R\). We write \(\mathbb R((t^{1/\infty}))_{\mathrm{alg}}\) for the subfield of Puiseux series algebraic over \(\mathbb R(t)\). By the Puiseux expansion of semialgebraic germs, the Hardy field of semialgebraic germs at \(0^+\) is identified as an ordered valued field with
\[
K:=\mathbb R((t^{1/\infty}))_{\mathrm{alg}};
\]
see \cite{MarkerMessmerPillay1996}.

For an ordered field \(A\) and an ordered abelian group \(B\), \(A((t^B))\) denotes the Hahn field of formal series with coefficients in \(A\) and well-ordered support in \(B\), equipped with the Hahn ordering. By the ordered Kaplansky embedding theorem, recalled in \cite[Chapter~3]{ADH2019} from Kaplansky's work \cite{Kaplansky1942,Kaplansky1944}, after passing to a sufficiently saturated elementary extension \(K^*\succeq K\), we may view \(K^*\) inside a Hahn field
\[
\mathbb R((t^{\mathbb R^*})),
\]
where \(\mathbb R^*\succ\mathbb R\) is \(\aleph_1\)-saturated and the value group is the additive ordered group of \(\mathbb R^*\).

The role of the Hahn field is only to represent cuts over \(K\). We do not use arbitrary elements of \(\mathbb R((t^{\mathbb R^*}))\) as concrete functions. The abstract curve selection lemma only records comparisons with elements of \(K\). Hence each coordinate of the abstract germ may first be replaced by another element of the Hahn field inducing the same cut over \(K\). The following reduction shows that, for this purpose, it is enough to consider elements of two explicit forms: either a rational-exponent generalized power series, or a polynomial in one Puiseux parameter together with one non-rational monomial. The first form is formal at this stage; in the concrete realization theorem below, we use it only when it represents a real-valued germ at \(0^+\).

After translating the limit point to the origin and reparametrizing the curve, the coordinates of the ordinary curve we want to construct are germs at \(0^+\) converging to \(0\). Thus, later, when these formal representatives are replaced by ordinary real germs, we choose representatives with the corresponding convergence property. The point of the next lemma is only the preliminary cut reduction over \(K\).

\begin{lemma}\label{Concrete form of abstract curve in an elementary extension}
Let \(\widetilde g\in\mathbb R((t^{\mathbb R^*}))\). Then there exists \(\widetilde f\in\mathbb R((t^{\mathbb R^*}))\) inducing the same cut over \(\mathbb R((t^{1/\infty}))_{\mathrm{alg}}\) as \(\widetilde g\), such that \(\widetilde f\) has one of the following forms:
\begin{itemize}
\item[(i)] \(\sum_{i\in\omega}a_i t^{q_i}\in\mathbb R[[t^{\mathbb Q}]]\), where \((q_i)_{i\in\mathbb N}\) is a strictly increasing sequence in \(\mathbb Q\);
\item[(ii)] \(P(t^{1/m})\pm t^r\) for some positive integer \(m\), some \(P(X)\in\mathbb R[X]\), and some \(r\in\mathbb R^*\setminus\mathbb Q\) with \(r>\ord(P)/m\).
\end{itemize}
\end{lemma}

\begin{proof}
Write \(\widetilde g=\sum_{\alpha\in I}a_\alpha t^{r_\alpha}\), where \(I\) is a well-ordered index set, \(a_\alpha\in\mathbb R^\times\), and the exponents \(r_\alpha\in\mathbb R^*\) are strictly increasing. If all exponents are rational and form a countable increasing sequence, then \(\widetilde g\) is of type \emph{(i)}.

Otherwise, let \(r\) be the first exponent of \(\widetilde g\) which is not rational. The initial segment before \(r\) consists of rational exponents. Since elements of \(\mathbb R((t^{1/\infty}))_{\mathrm{alg}}\) have Puiseux expansions with a common denominator, this rational initial segment determines, over the base field, a Puiseux polynomial. Thus the relevant initial part of \(\widetilde g\) has the form
\[
P(t^{1/m})+a_r t^r+\text{higher order terms},
\]
where \(P(X)\in\mathbb R[X]\), \(m\geq1\), \(a_r\in\mathbb R^\times\), and all remaining terms have exponent strictly larger than \(r\).

Assume first that \(a_r>0\). Set \(\widetilde f:=P(t^{1/m})+t^r\). Then \(\widetilde f\) has form \emph{(ii)}. Let \(h\in\mathbb R((t^{1/\infty}))_{\mathrm{alg}}\). The comparison between \(h\) and either \(\widetilde f\) or \(\widetilde g\) is decided by the first exponent at which the corresponding series differ. The series \(\widetilde f\) and \(\widetilde g\) have the same rational initial segment before \(r\), and their first nonrational term has the same positive sign. Since no element of \(\mathbb R((t^{1/\infty}))_{\mathrm{alg}}\) has a term of exponent \(r\), the sign of \(h-\widetilde f\) agrees with the sign of \(h-\widetilde g\). Therefore \(h<\widetilde f\) if and only if \(h<\widetilde g\), so \(\widetilde f\) and \(\widetilde g\) induce the same cut over \(\mathbb R((t^{1/\infty}))_{\mathrm{alg}}\).

If \(a_r<0\), the same argument gives the representative \(P(t^{1/m})-t^r\). Thus, in all cases, \(\widetilde g\) has a representative of one of the stated forms inducing the same cut over \(\mathbb R((t^{1/\infty}))_{\mathrm{alg}}\).
\end{proof}

We now compare the formal element obtained above with ordinary germs at \(0^+\). We do not need to identify the formal element with a real function. We only need to preserve its comparisons with algebraic Puiseux germs, measured against positive algebraic Puiseux germs.

Let
\[
K:=\mathbb R((t^{1/\infty}))_{\mathrm{alg}},
\]
identified with the Hardy field of semialgebraic germs at \(0^+\). We use the same notation for an element of \(K\) and for the corresponding semialgebraic germ. Let \(K^*\) be an elementary extension of \(K\) as an ordered field. If \(\widetilde f\in K^*\), then \(\widetilde f\) determines a metric cut over \(K\) by the comparisons
\[
|\widetilde f-g|<h,
\]
where \(g\in K\) and \(h\in K^{>0}\).

We shall also compare certain real germs at \(0^+\) with elements of \(K\). If \(f\) is a real-valued germ at \(0^+\) and \(g\in K\), we write \(|f-g|<h\) to mean that, for all sufficiently small \(t>0\), the inequality \(|f(t)-g(t)|<h(t)\) holds. Thus such a germ \(f\) determines a metric cut over \(K\) by the eventual comparisons of \(|f-g|\) with positive elements of \(K\). We say that a germ \(f\) realizes the same metric cut over \(K\) as an element \(\widetilde f\in K^*\) if, for every \(g\in K\) and every \(h\in K^{>0}\),
\[
|\widetilde f-g|<h
\quad\Longleftrightarrow\quad
|f-g|<h.
\]

Our goal is to replace the special representatives from Lemma~\ref{Concrete form of abstract curve in an elementary extension} by concrete real germs whenever possible. In case \emph{(i)}, the formal generalized power series is used only when it represents a real-valued germ at \(0^+\). In case \emph{(ii)}, each coordinate has the form
\[
\widetilde f=P(t^{1/m})\pm t^r,
\]
where \(m\in\mathbb N_{>0}\), \(P(X)\in\mathbb R[X]\), and \(r\in\mathbb R^*\setminus\mathbb Q\). After translating and reparametrizing the curve, we only need representatives which converge to \(0\) as \(t\to0^+\). Thus in this case we assume \(P(0)=0\) and \(r>0\). The polynomial part \(P(t^{1/m})\) is already a semialgebraic germ, so only the term \(\pm t^r\) has to be replaced by an ordinary real germ.

Since \(\mathbb R\) is Dedekind complete, the pair \(\mathbb R\preceq\mathbb R^*\) is tame. Hence, if \(r\) is \(\mathbb R\)-bounded, its standard part \(\st_{\mathbb R}(r)\) exists. We now choose, according to the position of the positive exponent \(r\) over \(\mathbb R\), a concrete comparison germ realizing the same metric cut over \(K\) as \(\pm t^r\).

\begin{lemma}\label{Concrete realization of special Hahn cuts}
Let
\[
\widetilde f=P(t^{1/m})\pm t^r,
\]
where \(m\in\mathbb N_{>0}\), \(P(X)\in\mathbb R[X]\), \(r\in(\mathbb R^*)^{>0}\setminus\mathbb Q\), and \(P(0)=0\). Then there exists a real-valued germ \(f\) at \(0^+\) which realizes the same metric cut over \(K\) as \(\widetilde f\). More precisely, \(f\) can be chosen as follows.

If \(r\in\mathbb R\), set \(f(t)=P(t^{1/m})\pm t^r\).

If \(r>\mathbb R\), set \(f(t)=P(t^{1/m})\pm e^{-1/t}\).

If \(r\) is \(\mathbb R\)-bounded, \(r>0\), and \(\st_{\mathbb R}(r)=0\), then \(P=0\), and we set
\[
f(t)=\mp\frac{1}{\log t}.
\]

Finally, suppose that \(r\) is \(\mathbb R\)-bounded and \(\st_{\mathbb R}(r)=s>0\). If \(r>s\), set \(f(t)=P(t^{1/m})\pm t^{s+t}\). If \(r<s\), set \(f(t)=P(t^{1/m})\pm t^{s-t}\).
\end{lemma}

\begin{proof}
It is enough to compare the final terms, since \(P(t^{1/m})\in K\). We treat the case \(\widetilde f=P(t^{1/m})+t^r\); the case with the minus sign is obtained by replacing the final comparison term by its negative.

Let \(g\in K\) and \(h\in K^{>0}\). We need to show that \(|\widetilde f-g|<h\) in \(K^*\) if and only if \(|f-g|<h\) eventually as germs at \(0^+\). Replacing \(g\) by \(g-P(t^{1/m})\), we may assume that \(P=0\). Since every nonzero element of \(K\) has a leading term \(ct^q\), with \(c\in\mathbb R^\times\) and \(q\in\mathbb Q\), comparisons over \(K\) are decided by comparisons with rational powers of \(t\).

If \(r\in\mathbb R\), then \(f(t)=t^r\), and the assertion is immediate.

Suppose \(r>\mathbb R\). Then \(t^r\) is smaller than every positive element of \(K\). The germ \(e^{-1/t}\) has the same property: for every \(q\in\mathbb Q\), one has \(e^{-1/t}\prec t^q\). Hence \(t^r\) and \(e^{-1/t}\) have the same comparisons with every element of \(K\) and every positive element of \(K\).

Suppose that \(r\) is \(\mathbb R\)-bounded, \(r>0\), and \(\st_{\mathbb R}(r)=0\). Then \(t^r\) tends to \(0\), but more slowly than every positive rational power of \(t\). Also \(-1/\log t>0\) for sufficiently small \(t>0\), and for every \(q\in\mathbb Q_{>0}\),
\[
t^q\prec -\frac{1}{\log t}\prec 1.
\]
Thus \(-1/\log t\) has the same comparisons over \(K\) as \(t^r\).

Finally suppose that \(r\) is \(\mathbb R\)-bounded and \(\st_{\mathbb R}(r)=s>0\). If \(r>s\), then for every \(q\in\mathbb Q\), the comparison between \(r\) and \(q\) agrees with the eventual comparison between \(s+t\) and \(q\). Hence \(t^r\) and \(t^{s+t}\) have the same comparisons with all rational powers of \(t\), and therefore with all elements of \(K\). If \(r<s\), the same argument applies with \(s-t\) in place of \(s+t\).

This proves the assertion in the plus case. The minus case follows by applying the same argument to the negatives of the final comparison terms.
\end{proof}

\begin{theorem}\label{Concrete realization of abstract CODF curve}
Let \(\mathcal M=(\mathbb R,\delta)\models\CODF\), let \(C\subseteq\mathbb R^k\) be a \(\mathcal C^\infty\) \(\delta\)-cell, and let \(a\in\cl_\delta(C)\). After translating \(a\) to \(0\) and reparametrizing the curve, let \(\widetilde f=(\widetilde f_1,\ldots,\widetilde f_N)\) be the abstract germ obtained from Theorem~\ref{Abstract curve selection lemma}, viewed over \(K=\mathbb R((t^{1/\infty}))_{\mathrm{alg}}\).

Assume that every coordinate of type \emph{(i)} in Lemma~\ref{Concrete form of abstract curve in an elementary extension} represents a real-valued germ at \(0^+\). Then there is a tuple of real-valued germs \(f=(f_1,\ldots,f_N)\) at \(0^+\), with each \(f_j(t)\to0\), such that \(f\) realizes the same coordinatewise metric cut over \(K\) as \(\widetilde f\).

Moreover, each coordinate \(f_j\) may be chosen in one of the following forms:
\begin{itemize}
\item[(i)] a rational-exponent series \(f_j(t)=\sum_{i\in\omega}a_i t^{q_i}\), where \((q_i)_{i\in\mathbb N}\) is strictly increasing in \(\mathbb Q\) and the series represents a real-valued germ;
\item[(ii)] one of the concrete germs listed in Lemma~\ref{Concrete realization of special Hahn cuts}.
\end{itemize}
\end{theorem}

\begin{proof}
Apply Lemma~\ref{Concrete form of abstract curve in an elementary extension} to each coordinate of \(\widetilde f\). In case \emph{(i)}, by the additional assumption, the coordinate is represented by a rational-exponent series with a real-germ interpretation, and we take that germ. In case \emph{(ii)}, the coordinate has the form \(P(t^{1/m})\pm t^r\) with positive exponent after translating and reparametrizing the curve, and Lemma~\ref{Concrete realization of special Hahn cuts} gives a real germ realizing the same metric cut over \(K\). Applying this construction coordinatewise gives the desired tuple \(f\).
\end{proof}

\bibliographystyle{alpha}
\bibliography{references}

\end{document}